\documentclass[AMA,LATO1COL]{my-WileyNJD-v2}

\articletype{Research Article}

\received{26-Jul-2021}
\accepted{17-Sep-2021}

\usepackage{amsmath,amssymb}

\begin{document}

\title{Controlling crop pest with a farming awareness based integrated approach and optimal control}

\author[1]{Teklebirhan Abraha}

\author[2]{Fahad Al Basir}

\author[1]{Legesse Lemecha Obsu}

\author[3]{Delfim F. M. Torres*}

\authormark{T. Abraha \textsc{et al.}}

\address[1]{\orgdiv{Department of Mathematics}, 
\orgname{Adama Science and Technology University}, 
\orgaddress{\state{Adama}, 
\country{Ethiopia}}}

\address[2]{\orgdiv{Department of Mathematics}, 
\orgname{Asansol Girls' College}, 
\orgaddress{\state{West Bengal 713304}, 
\country{India}}}

\address[3]{\orgdiv{R\&D Unit CIDMA, Department of Mathematics}, 
\orgname{University of Aveiro}, 
\orgaddress{\state{3810-193 Aveiro}, 
\country{Portugal}}}

\corres{*Delfim F. M. Torres, 
R\&D Unit CIDMA, Department of Mathematics,
University of Aveiro, 3810-193 Aveiro, Portugal.
\email{delfim@ua.pt}}

\abstract[Summary]{We investigate a mathematical model in crop pest controlling, 
considering plant biomass, pest, and the effect of farming awareness. 
The pest population is divided into two compartments: susceptible pests and infected pests. 
We assume that the growth rate of self-aware people is proportional to the density 
of susceptible pests existing in the crop arena. Impacts of awareness are modeled through 
the usual mass action term and a saturated term. It is further assumed that self-aware people 
will adopt chemical and biological control methods, namely integrated pest management. 
Bio-pesticides are costly and require a long-term process, expensive to impose. 
However, if chemical pesticides are introduced in the farming system along with bio-pesticides, 
the process will be faster as well as cost-effective. Also, farming knowledge is equally important. 
In this article, a mathematical model is derived for controlling crop pests through 
an awareness-based integrated approach. In order to reduce the negative effects of pesticides, 
we apply optimal control theory.}


\jnlcitation{\cname{%
\author{Abraha T.}, 
\author{Al Basir F.}, 
\author{Obsu L. L.}, and 
\author{Torres D. F. M.}} (\cyear{2021}), 
\ctitle{Controlling crop pest with a farming awareness based integrated approach and optimal control}, 
\cjournal{Comp. and Math. Methods}, \cvol{2021}.}

\maketitle

\footnotetext{\textbf{Abbreviations:} 
IPM, Integrated Pest Management; PMP, Pontryagin Minimum Principle;
NPV, Nuclear Polyhedrosis Virus}


\section{Introduction}
\label{sec1}

Problems connected with pests have become evident around the world 
as cultivation began. World's food supply is being wasted due to 
the cause of pests in agriculture. On the other hand, major side-effects 
of synthetic pesticides on the environment, human health, 
and biodiversity, are generating widespread concerns. Thus, farmers' awareness 
of the risk of synthetic pesticides uses is one of the important factors to consider. 
The use of biological contents to protect crops against pests needs indigenous 
knowledge to implement such contents in pest management \cite{TAbraha,Abraha2021}.

There are several good modeling studies on pest control. For example, Chowdhury et al. 
\cite{NS,EC19} have proposed and analyzed mathematical models for biological pest control 
using the virus as a controlling agent. In fact, all eco-epidemic models with susceptible prey, 
infected prey, and predators, can be used to discuss the nature of the susceptible pest, 
infected pest, and their predators \cite{JTB}. Zhang et al. \cite{Zhang} used a delayed 
stage-structured epidemic model for pest management strategy. Wang and Song \cite{Wang08} 
used mathematical models to control a pest population by infected pests. However, they did 
not use the influence of the predator populations on their works.

Many researchers utilize mathematical models for pest control in order to study different 
aspects of pest management policies with probable outcomes for improved applications, 
using system's analysis within the mathematical paradigm. Most of them suggest using 
chemical pesticides \cite{Ghosh,Kar}. However, it is recorded that chemical pesticides 
have resulted in pest resurgence, acute and chronic health problems, and environmental 
pollution \cite{Bhattacharya04}. Thus, to resolve this type of problem, the concept of 
IPM is becoming more popular among researchers with increasing application 
in the field by marginal farmers.

In this paper, we formulate a mathematical model, 
incorporating the farming awareness based integrated approach. 
The main focus is to compare the basic advantage of favoring 
the biological and combined strategy to minimize 
the pest problem and predict new insights on the
pest management, in general.  In order to reduce the negative effects 
of pesticides, we apply an optimal control approach. The dynamic 
of the system, without application of control, is analyzed 
through stability and bifurcation theory. Then, we formulate 
a three control parameter optimal control problem and solve it 
by applying PMP to find out the optimal level of both pesticide 
and the advertisement cost for cost effectiveness and
minimizing the negative effect due to pesticides.
Numerical simulations illustrate the analytical results. 
Finally, we discuss the outcomes with a conclusion.


\section{Description and model formulation}
\label{sec2}

We consider four populations into our mathematical model, 
namely plants biomass $X(t)$, susceptible pest $S(t)$, 
infected pest $I(t)$, and level of awareness $A(t)$.
The following assumptions are made to formulate the mathematical model:
\begin{itemize}
\item Under influence of bio-pesticides, healthy pest population becomes infected.
Infected pest can attack the plant but the rate is very lesser than susceptible
pest. We assume that infected pest can consume the plant biomass following 
a Holling type II response, whereas susceptible pest consume 
following a Holling type I response function.
	
\item Due to the finite size of crop field, we assume logistic growth
for the density of crop biomass, with net growth rate $r$ and carrying capacity $K$.
	
\item Susceptible attacks the crop, thereby causing considerable crop reduction.
If we infect the susceptible pest by pesticides, then the attack by pest
can be controlled. Here we assume that aware farmers will adopt biological pesticides
for the control of the crop pest, as it has less side effects and is also environment friendly.
Biopesticides are used to infect the healthy pest. Infected pest has an additional 
mortality due to infection.
	
\item Let $\alpha$ be the consumption rate of pests. 
There is a pest infection rate, $\lambda$, because of aware human 
interactions and activity such as use of biopesticides (e.g., NPV), 
modeled via the usual mass action term $\lambda A S$. We denote by $d$ the natural 
mortality rate of pest and by $\delta$ the additional mortality rate 
of infected pest due to aware people activity.
	
\item It is assumed that the level of awareness will increase at a rate $\omega$, 
proportional to the number of pests per plant noticed in the farming system. 
There could be fading of interest in this exploitation. We let $\eta$ 
be the rate of fading of interest of aware people.

\item To speed up the pest control process, chemical synthetic pesticides are introduced. 
It causes additional death to both susceptible and infected pest populations. Following
\cite{EC18}, we model the situation by the terms $\frac{\gamma SA}{1+A}$ 
and $\frac{\gamma IA}{1+A}$, respectively.
\end{itemize}

Based on the above assumptions, we have the following mathematical model:
\begin{equation}
\label{eq1}
\begin{cases}
\displaystyle \frac{dX}{dt} 
=r X\left(1-\frac{X}{K}\right)-\alpha XS-\frac{\phi \alpha XI}{a+X},\\[0.3cm]
\displaystyle \frac{dS}{dt} 
=m_1\alpha XS-\lambda AS-dS-\frac{\gamma SA}{1+A} ,\\[0.3cm]
\displaystyle \frac{dI}{dt} 
=\frac{m_2\phi \alpha XI}{a+X}+\lambda AS-(d+\delta)I-\frac{\gamma IA}{1+A},\\[0.3cm]
\displaystyle \frac{dA}{dt} =\omega+\sigma(S+I) - \eta A,
\end{cases}
\end{equation}
subject to the initial conditions 
\begin{equation}
\label{eq:ic}
X(0)\geq0,\quad S(0)\geq0, \quad I(0)\geq0, \quad A(0)\geq 0.
\end{equation}
Here $\alpha$ is the attack rate of pests on crop. 
Infected pests can also attack the crop but with a lower rate, $\phi\alpha$,
with $\phi<1$; $a$ is the half saturation constant, $m_1$ and $m_2$ are the ``conversion efficiency'' 
of susceptible and infected pests, respectively, i.e., they measure how efficiently can the pests 
utilize plant resource. Since pesticide affected pests have lowered efficiency, $m_1>m_2$, 
and $\gamma$ denotes the increase of level from global advertisement by radio, TV, etc.
It is natural to assume that all the parameters of model \eqref{eq1} are positive.


\section{Positivity of solutions and the invariant region}
\label{sec3}

Existence and positivity of the solutions are the main properties 
that system \eqref{eq1} must satisfy for the model to be
well-posed. Such properties are proved in this section. They describe the range 
in which the solution of the equations is biologically important. 

The feasible region is given by		
\[
\Omega=\{(X,S,I,A)\in\mathbb{R}_+^{4}\colon X\geq0, S\geq0,I\geq0,A\geq0\}.
\]
To show that the first two equations of the system \eqref{eq1} are positive, 
we use the following lemma. 

\begin{lemma}
\label{lemma1}
Any solution of the differential equation 
$\frac{dX}{dt}=X\psi(X, Y )$ is always positive.
\end{lemma}

\begin{proof}
A differential equation of the form $\frac{dX}{dt}=X\psi(X, Y )$ 
can be written as $\frac{dX}{X}=X\psi(X, Y )dt$. Integrating, 
we can write that $\ln X=C_0+\int\psi(X,Y)dt$, i.e., we have
$X=C_1e^{\int\psi(X,Y)dt}>0$ for $C_1>0$. 
\end{proof}

\begin{theorem}[non-negativeness of the solutions]
The solutions of system \eqref{eq1} subject to given non-negative
initial conditions \eqref{eq:ic} remain non-negative for all $t>0$.
\end{theorem}

\begin{proof}
Let \((X(t),S(t),I(t),A(t))\) be a solution of system\eqref{eq1} 
with its initial conditions \eqref{eq:ic}. We use Lemma~\ref{lemma1}
to prove the positivity of the equations in the system. 
Let us consider \(X(t)\) for \(t\in[0,T]\). We obtain, 
from the first equation of system \eqref{eq1}, that 
\[
{\frac{dX}{dt}}=X\,\left[r\,\left(1-\frac{X}{K}\right)-\alpha\, 
S-{\frac{\phi\alpha I}{a+X}}\right]
\Rightarrow 
{\frac{dX}{X}}=\left[r\,\left(1-\frac{X}{K}\right)
-\frac{\alpha S+\phi\alpha I}{c+X}\right]dt.
\]
Hence,
\[
\ln X=D_0+\int_{0}^{t} \left[r\,\left(1-\frac{X}{K}\right)
-\alpha S-\frac{\phi\alpha I}{a+X}\right]du,
\]
so that
\[
X(t)=D_1e^{\int_{0}^{t}\left[r\,\left(1-\frac{X}{K}\right)
-\alpha S-\frac{\phi\alpha I}{a+X}\right]du}>0
\quad (\because D_1=e^{D_0}).
\]
From the second equation of system \eqref{eq1}, we have
\[
\frac{dS}{dt}=S\left[m_1\alpha X-\lambda A-d-\frac{\gamma A}{1+A}\right]
\Rightarrow \frac{dS}{S}=\left[m_1\alpha X-\lambda A-d-\frac{\gamma A}{1+A}\right]dt.
\]
Hence,
\[
\ln S=K_0+\int_{0}^{t} \left(m_1\alpha X-\lambda A-d-\frac{\gamma A}{1+A}\right)du,
\]
so that
\[
S(t)=K_1e^{\int_{0}^{t} \left(m_1\alpha X-\lambda A-d-\frac{\gamma A}{1+A}\right)du}>0.
\]
To show that $I$ and $A$ are non-negative, consider the following sub-system of \eqref{eq1}:
\begin{equation}
\label{subsystem}
\begin{aligned}
\frac{dI}{dt}&=\frac{m_2\phi\alpha XI}{a+X}+ \lambda AS-(d+\delta)I-\frac{\gamma IA}{1+A},\\
\frac{dA}{dt}&=\omega+\sigma(S+I)-\eta A.
\end{aligned}
\end{equation}
To show the positivity of \(I(t)\), we do the proof by contradiction. 
Suppose there exists \(t_0\in(0,T)\)  such that $I(t_0)=0$, 
$I'(t_0)\leq0$, and \(I(t)>0\) for \(t\in[0,t_0)\). Then, \(A_0>0\) 
for \(t\in[0,t_0)\). If this is not to be the case, then there exists 
\(t_1\in[0,t_0)\) such that $A(t_1)=0$, $A'(t_1)\leq 0$ 
and \(A(t)>0\) for \(t\in[0,t_0)\). 
Integrating the third equation of the system \eqref{eq1} gives
\begin{equation*}
\begin{aligned}
I(t)&=I(0)\,\exp\left(\int_{0}^{t}\left( \frac{m_2\,\phi\,\alpha\,X(\tau)}{a+X(\tau)}
-\frac{\gamma\,\,A(\tau)}{1+A(\tau)}\right)d\tau-(d+\delta)t\right)\\
&+\left[\exp\left(\int_{0}^{t}\left( \frac{m_2\,\phi\,\alpha\,X(\tau)}{a+X(\tau)}
-\frac{\gamma\,\,A(\tau)}{1+A(\tau)}\right)d\tau-(d+\delta)t\right)\right]\\
&\times\left[\int_{0}^{t}\lambda\,A(\tau)S(\tau)d\tau\,
\exp\left((d+\delta)t-\int_{0}^{t}\left( \frac{m_2\,\phi\,\alpha\,X(\tau)}{a+X(\tau)}
-\frac{\gamma\,\,A(\tau)}{1+A(\tau)}\right)d\tau\right)\right]>0, 
\quad\mbox{for}\quad t\in[0,t_1].
\end{aligned}
\end{equation*}
Then, \(A'(t_1)=\gamma+\sigma(S(t_1)+I(t_1))>0\). 
This is a contradiction. Hence, \(I(t)>0\) for all \(t\in[0,t_0)\).
Finally, from the second equation of subsystem \eqref{subsystem}, 
we have
\[
\frac{dA}{dt}=\omega+\sigma(S+I)-\eta A.
\] 
Integration gives
\[
A(t)=A(0)\,e^{\eta t}+e^{\eta t}\,
\int_{0}^{t}\left(\omega+\sigma(S(\tau)
+I(\tau))\right)e^{-\eta t}dt>0,
\]
that is, \(A(t)>0\) for all \(t\in(0,T)\). 
\end{proof}

\begin{theorem}[boundedness of solutions]
\label{bounded}
Every solution of system \eqref{eq1} that start 
in $\mathbb{R}^{4}_+$ is uniformly bounded 
in the region $\mathcal{V}$ defined by
\[
\mathcal{V}=\left\{(X,S,I,A)\in{\mathbb R}_{+}^{4}\colon 0<X+S+I\leq\frac{L}{d},0<A
\leq\frac{\omega d+\sigma L}{\eta d}\right\},\quad\mbox{with\quad}L=\frac{K(r+d)^2}{4r}.
\]
\end{theorem}

\begin{proof}
We choose $m=\max\left\{m_1,m_2\right\}=m_1$, 
because in our assumptions we assume that $m_1>m_2$.
Now, at any time $t$, let $W=X+\frac{1}{m_1}S+\frac{1}{m_1}I$. 
Then the time derivative of $W$ along the solution of system \eqref{eq1} is given by
\begin{eqnarray*}
\frac{dW(t)}{dt}
&=& r X\left[1 - \frac{X}{K}\right] - \alpha XS- \frac{\phi \alpha XI}{a+X}
+\alpha XS-\frac{\lambda }{m_1}AS+\frac{\lambda }{m_1}AS\\
& +& \frac{m_2}{m_1}\frac{\phi\alpha XI}{a+X}-\frac{d}{m_1}S
- \left(\frac{d+\delta}{m_1}\right)I-(S+I)\frac{\gamma}{m_1}\frac{A}{1+A}\\
&=&r X\left[1 - \frac{X}{K}\right] - \left(\frac{m_1-m_2 }{m_1}\right)\frac{\phi\alpha XI}{a+X}\\
&-&\left(\frac{d+\delta}{m_1}\right)I-\frac{d}{m_1}S-(S+I)\frac{\gamma}{m_1}\frac{A}{1+A}\\
&\leq& r X\left[1-\frac{X}{K}\right] -\frac{d}{m_1}S-\frac{d}{m_1}I\\
&=&r X(1-\frac{X}{K})-\frac{d}{m_1}(S+I)-dX+dX\\
&=&r X\left(1-\frac{X}{K}\right)-\left(\frac{d}{m_1}S+\frac{d}{m_1}I+dX\right)+dX.
\end{eqnarray*}
Then, we have from the above that
\[
\frac{dW}{dt}\leq r X\left(1-\frac{X}{K}\right)-dW+dX,
\]
that is,
\[
\frac{dW}{dt}+dW\leq r X\left(1-\frac{X}{K}\right)+dX
=(r+d)X-\frac{rX^2}{K}=:\Phi(X).
\]
Now, $\Phi(X) $ is a concave parabola for which its maximum value is attained 
at the vertex whose abscissa is $X_v=\frac{K(r+d)}{2r}$. Therefore,
it follows
\[
\Phi(X)\leq \Phi(X_v)=\frac{1}{4r}K(r+d)^2=\Phi^*.
\]
Thus, we have a constant $L=\frac{K(r+d)^2}{4r}$ such that
\[
\frac{dW}{dt}+dW\leq L.
\]
To solve this, we apply the differential inequality
\begin{equation*}
\begin{split}
e^{dt}\left(\frac{dW}{dt}+dW\right) \leq e^{dt}L
&\Rightarrow \frac{d}{dt}\left(We^{dt}\right) \leq e^{dt}L\\
&\Rightarrow We^{dt} \leq \int(Le^{dt})dt+C\\
&\Rightarrow W(X,S,I) \leq e^{-dt}\left[dLe^{dt}\right]+Ce^{-dt}\\
&\Rightarrow W(X(0),S(0),I(0)) \leq \frac{L}{d}+C\\
&\therefore W(X,S,I) \leq\frac{L}{d}\left(1-e^{-dt}\right)
+W(X(0),S(0),I(0))e^{-dt}.
\end{split}
\end{equation*}
Hence, we get
\[
0<W(X,S,I)\leq\frac{L}{d}(1-e^{-dt})+W(X(0),S(0),I(0))e^{-dt}
\]
and, for $t\rightarrow\infty$, we have
\[
0<X+S+I\leq\frac{L}{d}.
\]
From the fourth equation of system \eqref{eq1}, we have
\begin{eqnarray*}
\frac{dA}{dt} &=&\omega+ \sigma(S+I) - \eta A\\
&\leq& \omega+\sigma\left(\frac{Lm_1}{d}\right)-\eta A\\
&=&\frac{d\omega+\sigma L}{d}-\eta A\\
&\Leftrightarrow& \frac{dA}{dt}+\eta A\leq \frac{d\omega+\sigma L}{d}.
\end{eqnarray*}
Again, applying the method of differential inequality, we have
\begin{equation*}
\begin{split}
e^{\eta t}\left(\frac{dA}{dt}+\eta A\right)
\leq e^{\eta t}\left(\frac{d\omega+\sigma  L}{d}\right)
&\Rightarrow \frac{d}{dt}\left(Ae^{\eta t}\right)
\leq e^{\eta t}\left(\frac{d\omega+\sigma  L}{d}\right)\\
&\Rightarrow Ae^{\eta t} 
\leq \int\left(\frac{d\omega+\sigma L}{d}\right)e^{\eta t}dt+C\\
&\Rightarrow A(t) \leq \left(\frac{d\omega+\sigma L}{\eta d}\right)+Ce^{-\eta t}\\
&\Rightarrow A(0) \leq \frac{\omega d+L\sigma}{\eta d}+C.
\end{split}
\end{equation*}
This results in
\[
0<A\leq\frac{\omega d+\sigma L}{\eta d}+Ce^{-\eta t}.
\]
Thus, for $t\rightarrow\infty$, we obtain that
\[
0<A\leq \frac{\omega d+\sigma L}{\eta d}.
\]
Hence, all solutions of \eqref{eq1} originating 
in $\mathbb{R}^{4}_+$ are confined to the region
\[
\mathcal{V}=\left\{(X,S,I,A)\in\mathbb{R}^{4}_+\colon 0<X+S+I\leq\frac{L}{d}
+\epsilon,0<A\leq \frac{\omega d+\sigma L}{\eta d}\right\}
\]
for any $\epsilon>0$ and for $t\rightarrow\infty$.
Thus, the system \eqref{eq1} is always uniformly bounded.
\end{proof}


\section{Equilibria assessment} 
\label{sec4}

To get the fixed points of our system, 
we put the right hand sides of system \eqref{eq1} equal to zero:
\begin{equation}
\label{eq3}
\begin{cases}
r X\left(1 - \frac{X}{K}\right) - \alpha XS - \frac{\phi \alpha XI}{a+X}=0,\\
m_1\alpha XS - \lambda AS - dS - \frac{\gamma SA}{1+A}=0 ,\\
\frac{m_2\phi \alpha XI}{a+X} + \lambda AS - (d+\delta)I- \frac{\gamma IA}{1+A}=0,\\
\omega + \sigma(S+I) - \eta A=0.
\end{cases}
\end{equation}
We conclude that system \eqref{eq1} has five possible equilibrium points,
denoted by $E_i$, $i = 0, 1, 2, 3$, and $E^*$:
\begin{itemize}
\item[(i)] The axial equilibrium point $E_0 = \left(0,0,0,\frac{\omega}{\eta}\right)$,
which always exists.
\item[(ii)] The pest free equilibrium point $E_1=\left(K,0,0,\frac{\omega}{\eta}\right)$, 
which, again, always exists.
\item [(iii)] The boundary equilibrium point $E_2=\left(0,S_1,I_1,A_1\right)$,
where 
\[
S_1=\frac{\left((d+\delta+\gamma)A+d+\delta\right)(\eta A-\omega)}{\left(A^2\lambda+(d+\delta
+\gamma+\lambda)A+d+\delta\right)\sigma},
\]
\[
I_1=\frac{(A\eta-\omega)\lambda A(1+A)}{(A^2\lambda+(d+\delta+\gamma+\lambda)A+d+\delta)\sigma},
\]
and $A_1$ is the positive root of the equation \(\lambda A^2+(\gamma+d+\lambda)A+d=0\). 
Unfortunately, this quadratic equation has no positive roots and, hence, 
such an equilibrium does not occur.

\item[(iv)] The healthy pest free equilibrium point \(E_3=\left(\bar{X},0,\bar{I},\bar{A}\right)\),
where $\bar{X}$, $\bar{I}$ and $\bar{A}$ are computed as follows.
If we set $S=0$ in system \eqref{eq3}, then 
\begin{equation}
\label{eq4}
\begin{cases}
r X\left[1 - \frac{X}{K}\right] - \frac{\phi \alpha XI}{a+X}=0,\\
\frac{m_2 \phi\alpha XI}{a+X} - (d+\delta+\frac{\gamma A}{1+A})I=0,\\
\omega+ \sigma I - \eta A=0.
\end{cases}
\end{equation}
From the first equation of the nonlinear system \eqref{eq4}, we have
\[
r \left[1 - \frac{X}{K}\right] - \frac{\phi \alpha I}{a+X}=0
\Rightarrow \frac{\phi\alpha I}{a+\bar{X}}=\frac{r(K-\bar{X})}{K}
\]
\[
\Rightarrow \phi\alpha I=\frac{r(a+\bar{X})(K-\bar{X})}{K}
	\Rightarrow \bar{I}=\frac{r(a+\bar{X})(K-\bar{X})}{\phi\alpha K}
\]
and, from the third equation  of system \eqref{eq4}, we get
\[
\omega+\sigma I-\eta A=0 \Rightarrow \eta A
=\omega+\sigma I\Rightarrow \bar{A}
=\frac{\omega+\sigma \bar{I}}{\eta}
=\frac{r(a+\bar{X})(K-\bar{X})\sigma+\omega K\alpha\phi}{K\alpha\eta\phi}.
\]
Finally, solving for $X$ from the second equation of system \eqref{eq4}, 
we see that $\bar{X}$ is the positive root of equation
\begin{equation}
\label{charc1}
X^3+a_1X^2+a_2X+a_3=0,
\end{equation}
where
\begin{equation*}
\begin{aligned}
a_1&={\frac { \left( \alpha\,\phi\,m_{{2}}-d-\delta-\gamma \right) K-a
\left( \alpha\,\phi\,m_{{2}}-2\,d-2\,\delta-2\,\gamma \right) }{
-\alpha\,\phi\,m_{{2}}+d+\delta+\gamma}},\\
a_{2}&={\frac { \left(  \left( \alpha\,\phi\,m_{{2}}-2\,d-2\,\delta-2\,
\gamma\right) K+a \left( d+\delta+\gamma \right)  \right) a}{-\alpha\,
\phi\,m_{{2}}+d+\delta+\gamma}}
-{\frac {K \left(  \left( -\alpha\,\phi\,m_{{2}}+d+\delta+\gamma
\right) \omega+\eta\, \left( -\alpha\,\phi\,m_{{2}}+d+\delta \right) 
\right) \phi\,\alpha}{r\sigma\, \left( -\alpha\,\phi\,m_{{2}}+d+
\delta+\gamma \right) }},\\
a_{3}&=-{\frac {K{a}^{2}r \left( d+\delta+\gamma \right) \sigma+K\phi\,a
\alpha\, \left(  \left( d+\delta+\gamma \right) \omega+\eta\, \left( d
+\delta \right)  \right) }{r\sigma\, \left( -\alpha\,\phi\,m_{{2}}+d
+\delta+\gamma \right)}}.
\end{aligned}
\end{equation*}
The model system \eqref{eq1} may have one or more healthy pest free equilibrium points 
$E_3$, depending on the positive solutions of equation\eqref{charc1}. The healthy pest 
free equilibrium point \(E_3\) exists only if the equation \eqref{charc1} 
has a positive root \(\bar{X}\) and \(K-\bar{X}>0\).

\item[(v)] Our model system \eqref{eq1} has an equilibrium point
in the presence of pest, $X(t)\geq0$, $S(t)\geq0$, $I(t)\geq0$, $A(t)\geq0$,
called the interior or coexistence or endemic equilibrium point, 
which is denoted by $E^*=(X^*,S^*,I^*,A^*)\neq0$. 
Note that $E^*$ is the steady state solution where pest persist in the crop biomass population. 
It is obtained by setting each equation of system \eqref{eq1} equal to zero, that is,
\[
\frac{dX}{dt}=\frac{dS}{dt}=\frac{dI}{dt}=\frac{dA}{dt}=0.
\]
From the second equation of system \eqref{eq3}, we get
\(\left(m_1 \alpha X-\lambda A -d-\frac{\gamma A}{1+A}\right)S=0\), that is, 
$m_1 \alpha X-\lambda A -d-\frac{\gamma A}{1+A}=0$,
from which we obtain
\[
X^*=\frac{\lambda A^2+(d+\lambda+\gamma)A+d}{m_1\alpha(1+A)}.
\]
From the first equation  of system \eqref{eq3}, we have
\begin{equation}
\label{star1}
\begin{split}
r\left(1-\frac{X}{K}\right)-\alpha S-\frac{\phi \alpha I}{a+X}=0
&\Rightarrow \alpha S+\frac{\phi\alpha I}{a+X}=r\left(1-\frac{X}{K}\right)=\frac{r(K-X)}{K}\\
&\Rightarrow \alpha S+\frac{\phi\alpha I}{a+X}=\frac{r(K-X)}{K}\\
&\Rightarrow \alpha(a+X)S+\phi\alpha I=\frac{r(K-X)(a+X)}{K}
\end{split}
\end{equation}
and from the last equation  of system \eqref{eq3} we get
\begin{equation}
\label{star2}
\alpha(a+X)S+\phi\alpha I=\frac{r(K-X)(a+X)}{K}.
\end{equation}
Solving equations \eqref{star1} and \eqref{star2} simultaneously, one obtains
\begin{equation*}
\begin{aligned}
S^*={\frac {r \left( a+X \right)  \left( K-X \right) \sigma
-K\phi\,\alpha\, \left( \eta\,A-\omega \right) }{K\alpha\,
\sigma\, \left( X+a-\phi\right) }},\\
I^*={\frac { \left( a+X \right)  \left( AK\alpha\,\eta-K\alpha\,\omega-Kr
\sigma+\sigma\,rX \right) }{K\alpha\,\sigma\, \left( X+a-\phi \right) }}.
\end{aligned}
\end{equation*}
Therefore,
\(E^{*}=(X^{*},S^{*},I^{*},A^{*})\) is the coexistence steady state with
\begin{equation*}
\begin{split}
X^{*}&={\frac{\lambda\, A^{2}+\left(d+\lambda+\gamma\right)\,A+d}{m_1\,\alpha\left(1+A\right)}},\\
S^{*}&=\frac{r(a+X)(K-X)\sigma-K\phi\alpha(A\eta-\omega)}{K\sigma\alpha(X+a-\phi)},\\
I^{*}&=\frac{(a+X)(((A\eta-\omega)\alpha-\sigma r)K+\sigma rX)}{K\sigma\alpha(X+a-\phi)}
\end{split}
\end{equation*}
and $A^*$ a positive root of equation
\begin{equation}
\label{char2}
f(A)=A^6+a_1A^5+a_2A^4+a_3A^3+a_4A^2+a_5A+a_6=0,
\end{equation}
whose coefficients are given by
\begin{equation*}
\begin{split}
a_1&={\frac {3\,{\lambda}^{2}r\sigma-r\sigma\, \left(  \left(  
\left( K-a\right) m_{{1}}+m_{{2}}\phi \right) \alpha-3\,d
-\delta-3\,\gamma\right) \lambda}{{\lambda}^{2}r\sigma}}
+\frac {m_{{1}}\eta\, \left( \phi\, \left( m_{{1}}-m_{{2}} \right) 
\alpha+d+\delta+\gamma \right) {\alpha}^{2}}{{\lambda}^{2}r\sigma},\\
a_2&={\frac {3\,\sigma\,{\lambda}^{3}r-3\,\sigma\,r \left(  \left(  \left( 
K-a \right) m_{{1}}+m_{{2}}\phi \right) \alpha-3\,d-\delta-2\,\gamma
\right) {\lambda}^{2}}{\sigma\,{\lambda}^{3}r}}\\
&\quad -{\frac {m_{{1}}K \left( \phi\, \left( m_{{1}}-m_{{2}} \right) 
\left( \omega-3\,\eta \right) \alpha+\omega\, \left( d+\delta+\gamma
\right) + \left( -3\,d-3\,\delta-2\,\gamma \right) \eta+\sigma\,r
\left( m_{{1}}a-m_{{2}}\phi \right)  \right) {\alpha}^{2}}{{\lambda}^{2}r\sigma}}\\
&\quad +{\frac { \left( -2\, \left( K-a \right)  \left( d+\delta/2+\gamma
\right) m_{{1}}-2\,m_{{2}}\phi\, \left( d+\gamma \right)  \right) \alpha
+3\, \left( d+\gamma \right)  \left( d+2/3\,\delta+\gamma\right) }{{\lambda}^{2}}}\\
&\quad +{\frac {K\eta\, \left(  \left( a \left( d+\delta+\gamma \right) m_{{1}}
-m_{{2}}\phi\, \left( d+\gamma \right)  \right) \alpha
+ \left( d+\gamma \right)  \left( d+\delta+\gamma \right)  \right) 
m_{{1}}{\alpha}^{2}}{\sigma\,{\lambda}^{3}r}},\\
\end{split}
\end{equation*}
\begin{equation*}
\begin{split}
a_3&=-3\,{\frac {{m_{{1}}}^{2}K{\alpha}^{3} \left( \phi\, \left( \omega-
\eta \right) \lambda+1/3\,a \left( \omega\, \left( d+\delta+\gamma
\right) -3\,\eta\, \left( d+\delta+2/3\,\gamma \right)  \right) 
\right) }{\sigma\,{\lambda}^{3}r}}\\
&\quad +3\,{\frac {m_{{1}}K{\alpha}^{3}\phi\, \left(  \left( \lambda+d/3+
\gamma/3 \right) \omega- \left( \lambda+d+2/3\,\gamma \right) \eta
\right) m_{{2}}}{\sigma\,{\lambda}^{3}r}}\\
&\quad +3\,{\frac {m_{{1}}K \left( \sigma\,ar \left( \lambda+d/3+\delta/3+
\gamma/3 \right) m_{{1}}- \left( \lambda+d/3+\gamma/3 \right) \phi\,m_
{{2}}r\sigma+ \left( \omega/3-\eta/3 \right) {\gamma}^{2} \right) {
\alpha}^{2}}{{\sigma \lambda }^{3}r}}\\
&\quad -3\,{\frac {m_{{1}}K \left(  \left(  \left( 2/3\,\omega-4/3\,\eta
\right) d+ \left( 2/3\,\omega-\eta/3 \right) \lambda+1/3\,\delta\,
\left( \omega-2\,\eta \right)  \right) \gamma+ \left( \omega/3-\eta
\right) d \right) {\alpha}^{2}}{{\sigma \lambda }^{3}r}}\\
&\quad +{\frac {- \left( K-a \right)  \left( {d}^{2}+ \left( \delta+2\,\gamma
+6\,\lambda \right) d+\gamma\, \left( \delta+3\,\gamma+4\,\lambda
\right)  \right) m_{{1}}-\phi\, \left( {d}^{2}+6\,\lambda\,d
+3\,{\lambda}^{2} \right) m_{{2}}}{3\,{\lambda}^{3}r}}\\
&\quad +{\frac {{\lambda}^{3}+ \left( 9\,d+3\,\delta+3\,\gamma \right){\lambda}^{2}
+ \left( 3\,{\gamma}^{2}+ \left( 12\,d+4\,\delta \right) 
\gamma+9\,{d}^{2}+6\,d\delta \right) \lambda+ \left( d+\gamma \right)^{2}\left( 
d+\delta+\gamma \right) }{{\lambda}^{3}}},\\
a_4&={\frac { \left(  \left( a \left( \omega-\eta \right) d
+\phi\, \left( \omega-\eta/3 \right) \lambda+a\gamma+\delta\, 
\left( \omega-\eta\right)  \right) m_{{1}}- \left( \omega-\eta/3
+\lambda+2/3\,\omega\,\gamma \right) \phi\,m_{{2}} \right) 
K{\alpha}^{3}}{\sigma\,{\lambda}^{3}r}}\\
&\quad -3\,{\frac {m_{{1}}K \left( \sigma\,ar \left( \lambda+d+\delta\right) m_{{1}}
- \left( \lambda+d+2/3\,\gamma \right) \phi\,m_{{2}}r\sigma+\delta\, 
\left( \omega-\eta \right) +\delta\, \left( \omega-\eta/3 \right) 
\lambda \right) {\alpha}^{2}}{\sigma\,{\lambda}^{3}r}}\\
&\quad -3\,{\frac { \left(  \left( K-a \right)  \left( 6\,\lambda+3\,\delta
+4\,\gamma \right) d+ \left( {\lambda}^{2}+ \left( 3\,\delta+2\,
\gamma\right) \lambda+2\,\delta\,\gamma+{\gamma}^{2} \right) m_{{1}}+\phi\,
\left( 3\,{d}^{2}+ \left( 6\,\lambda+4\,\gamma \right) d \right) 
m_{{2}} \right) \alpha}{{\lambda}^{3}}}\\
&\quad -9\,{\frac {\sigma\, \left( {d}^{3}+ \left( 3\,\lambda+\delta+2\,
\gamma \right) {d}^{2}+ \left( {\lambda}^{2}+ \left( 2\,\delta+2\,
\gamma \right) \lambda+4/3\,\delta\,\gamma+{\gamma}^{2} \right) d+1/3
\,\delta\, \left( \lambda+\gamma \right) ^{2} \right) }{{\lambda}^{3}}},\\
a_5&=-{\frac {K \left(  \left(  \left( \lambda\,\phi+3\, \left( d+\delta
+\gamma/3 \right) a \right) \omega-a\eta\, \left( d+\delta \right) 
\right) m_{{1}}-m_{{2}} \left(  \left( \lambda+3\,d+\gamma \right) 
\omega-d\eta \right) \phi \right) {\alpha}^{3}m_{{1}}}{\sigma\,{
\lambda}^{3}r}}\\
&\quad -{\frac {m_{{1}} \left( \sigma\,ar \left( \lambda+3\,\delta+\gamma
\right) m_{{1}}-m_{{2}}r\phi\, \left( \lambda+3\,d \right) \sigma+
\left( \delta\, \left( 3\,\omega-\eta \right) +\omega\, \left( 
\lambda+2\,\gamma \right)  \right) d \right) K{\alpha}^{2}}{\sigma\,{
\lambda}^{3}r}}\\
&\quad -2\,{\frac { \left(  \left( K-a \right)  \left( 3/2\,{d}^{2}+ \left( 
\lambda+3/2\,\delta+\gamma \right) d+1/2\,\delta\, \left( \lambda+
\gamma \right)  \right) m_{{1}}+m_{{2}}\phi\,d \left( \lambda+3/2\,d
+\gamma \right)  \right) \alpha}{{\lambda}^{3}}}\\
&\quad + 3\,{\frac { \left( {d}^{2}+ \left( \lambda+\delta+\gamma \right) d
+2/3\,\delta\, \left( \lambda+\gamma \right)  \right) d}{{\lambda}^{3}}},\\
a_6&={\frac {m_{{2}}d\alpha\, \left( K{\alpha}^{2}\omega\,m_{{1}}+K\alpha\,
r\sigma\,m_{{1}}-dr\sigma \right) \phi}{\sigma\,{\lambda}^{3}r}}
-{\frac { \left(  \left( adm_{{1}}+a\delta\,m_{{1}} \right) \alpha+d
\left( d+\delta \right)  \right)  \left( K{\alpha}^{2}\omega\,m_{{1}}
+K\alpha\,r\sigma\,m_{{1}}-dr\sigma \right) }{\sigma\,{\lambda}^{3}r}}.
\end{split}
\end{equation*}
The coexistence equilibrium point \(E^{*}\) exists only if the characteristic 
equation \eqref{char2} has a positive root in \(A\) with 
\(A>\frac{\alpha\,\omega+r\,\sigma}{\alpha\,\eta}\).
\end{itemize}


\section{Stability of the equilibria}
\label{sec5}

The stability analysis is done by linearization 
of the non-linear system \eqref{eq1}. We write the Jacobian matrix $J$ at the fixed 
points of the system and compute the characteristic equation. 
Then, the stability of the equilibrium point is studied depending on the eigenvalues 
of the corresponding Jacobian, which are functions of the model parameters.
The Jacobian matrix for system \eqref{eq1}, at a steady 
state \((X,S,I,A)\), is given by
\begin{equation}
\label{eq:JM:g}
J(X,S,I,A) =
\left[ 
\begin{array}{cccc} 
r \left( 1-{\frac {2\,X}{K}} \right) -
\alpha\,S-{\frac {\phi\,\alpha\,a\,I}{ \left( a+X \right) ^{2}}}
&-\alpha\,X&-{\frac {\phi\,\alpha\,X}{a+X}}&0\\ \noalign{\medskip}m_{{1}}
\alpha\,S&m_{{1}}\alpha\,X-\lambda\,A-d-{\frac {\gamma\,A}{1+A}}&0&
-\lambda\,S-{\frac {\gamma\,S}{ \left( 1+A \right) ^{2}}}\\ 
\noalign{\medskip}{\frac {m_{{2}}\phi\,\alpha\,aI}{ \left( a+X
\right) ^{2}}}&\lambda\,A&{\frac {m_{{2}}\phi\,\alpha\,X}{a+X}}
-d-\delta-{\frac {\gamma\,A}{1+A}}&-\lambda\,S-{\frac {\gamma\,I}{ 
\left(1+A \right) ^{2}}}\\ 
\noalign{\medskip}0&\sigma&\sigma&-\eta
\end{array} 
\right]. 
\end{equation}

\begin{theorem}[stability of the crop-pest free equilibrium]
The system is always unstable around the crop-pest free equilibrium point $E_0$.
\end{theorem}

\begin{proof}
The Jacobian matrix \eqref{eq:JM:g} at the crop-pest free equilibrium $E_0$ is given by
\[
J\left(0,0,0,\frac{\omega}{\eta}\right) 
= \left[ 
\begin{array}{cccc} 
r&0&0&0\\ 
\noalign{\medskip}0&-{\frac {
\lambda\,\omega}{\eta}}-d-{\frac {\gamma\,\omega}{\eta+\omega}}&0&0\\ 
\noalign{\medskip}0&{\frac {\lambda\,\omega}{\eta}}&-d-\delta-{
\frac {\gamma\,\omega}{\eta+\omega}}&0\\ \noalign{\medskip}0&\sigma&
\sigma&-\eta
\end{array} 
\right],
\]
whose characteristic equation is 
\[
\left|\rho-J(E_0)\right|=\left| \begin {array}{cccc} {\rho}-r&0&0&0\\ 
\noalign{\medskip}0&{\rho}+\left({\frac{\lambda\,\omega}{\eta}}
+d+{\frac {\gamma\,\omega}{\eta+\omega}}\right)&0&0\\ 
\noalign{\medskip}0&-{\frac {\lambda\,\omega}{\eta}}&{\rho}
+\left(d+\delta+{\frac {\gamma\,\omega}{\eta+\omega}}\right)&0\\ 
\noalign{\medskip}0&-{\sigma}&-{\sigma}&{\rho}+\eta\end {array} \right|=0,
\]
that is,
\[
\left( \rho-r \right)  \left( \rho
+{\frac {\lambda\,\omega}{\eta}}+d
+{\frac {\gamma\,\omega}{\eta+\omega}} \right)  \left( \rho
+d+\delta+{\frac {\gamma\,\omega}{\eta+\omega}}\right)\left( 
\rho+\eta \right)=0.
\]
The corresponding eigenvalues are:
\[
\rho_1=r>0, \quad
\rho_2=-\eta<0, \quad
{\rho}_{3}=-\left({\frac{\lambda\,\omega}{\eta}}
+d+{\frac {\gamma\,\omega}{\eta+\omega}}\right)<0,\quad 
{\rho}_{4}=-\left(d+\delta+{\frac{\gamma\,\omega}{\eta+\omega}}\right)<0.
\]
Since one eigenvalue is positive, \(\rho_1=r>0\), 
the axial equilibrium \(E_0\) is always unstable.
\end{proof}

\begin{theorem}[stability of the pest free equilibrium]
The pest free steady state $E_1$ is locally asymptotically stable 
if the two critical parameters \(R_{0}\) and \(R_{1}\), 
\begin{equation}
\label{eq:R0:R1}
\begin{split}
R_{0}&:={\frac{m_{{1}}\alpha\,K\,\eta\,\left(\eta
+\omega\right)}{\lambda\,\omega\left(\eta
+\omega\right)+\eta\,\gamma\,\omega+d\eta\left(\eta+\omega\right)}},\\ 
R_{1}&:={\frac{m_{{2}}\phi\,\alpha\,K\,(\eta+\omega)}{(a+K)\,\left(d+\delta\right)\,
\left(\eta+\omega\right)+\left(a+K\right)\gamma\,\omega}},
\end{split}
\end{equation}
satisfy \(R_{0}<1\) and \(R_{1}<1\). Otherwise, $E_1$ is unstable.	
\end{theorem}

\begin{proof}
The Jacobian matrix $J(E_1)$, at the pest free equilibrium 
point $E_1=\left(K,0,0,\frac{\omega}{\eta}\right)$, is given by
\[
J\left(K,0,0,\frac{\omega}{\eta}\right)
=\left[ 
\begin{array}{cccc} 
-r&-K\alpha&-{\frac {\phi\,\alpha\,K}{a+K}}&0\\ 
\noalign{\medskip}0&m_{{1}}\alpha\,K-{\frac {\lambda\,\omega}{
\eta}}-d-{\frac {\gamma\,\omega}{\eta+\omega}}&0&0\\ 
\noalign{\medskip}0&{\frac {\lambda\,\omega}{\eta}}&{\frac {m_{{2}}
\phi\,\alpha\,K}{a+K}}-d-\delta-{\frac {\gamma\,\omega}{\eta+\omega}}&0\\ 
\noalign{\medskip}0&\sigma&\sigma&-\eta
\end{array}
\right].
\]
The characteristic equation in $\rho$ at $E_1$ is
\[
|\rho I-J(E_1)|
=\left|
\begin{array}{cccc} 
{\rho}+r&K\alpha&{\frac {\phi\,\alpha\,K}{a+K}}&0\\ 
\noalign{\medskip}0&{\rho}-m_{{1}}\alpha\,K
+{\frac {\lambda\,\omega}{\eta}}+d+{\frac {\gamma\,\omega}{\eta+\omega}}&0&0\\ 
\noalign{\medskip}0&-{\frac {\lambda\,\omega}{\eta}}
&{\rho}-{\frac {m_{{2}}\phi\,\alpha\,K}{a+K}}+d+\delta
+{\frac {\gamma\,\omega}{\eta+\omega}}&0\\ 
\noalign{\medskip}0&-\sigma&-\sigma&\rho+\eta
\end{array}
\right|=0,
\]
which gives
\[
\left( \rho+r \right)  \left( \rho-m_{{1}}\alpha\,K
+{\frac {\lambda\,\omega}{\eta}}+d+{\frac {\gamma\,\omega}{\eta+\omega}}\right) 
\left( \rho-{\frac {m_{{2}}\phi\,\alpha\,K}{a+K}}+d+\delta
+{\frac {\gamma\,\omega}{\eta+\omega}} \right)  \left( \rho+\eta \right) = 0. 
\]
Thus, the eigenvalues are $\lambda_{1}=-r$, 
$\lambda_{2}=m_{{1}}\alpha K-{\frac {\lambda\,\omega}{\eta}}
-d-{\frac {\gamma\,\omega}{\eta+\omega}}$,
$\lambda_{3}
={\frac {m_{{2}}\phi\,\alpha\,K}{a+K}}-d-\delta-{\frac {
\gamma\,\omega}{\eta+\omega}}$, and $\lambda_{4}=-\eta$.
We have that $(E_1)$ is locally asymptotically stable if all 
the four eigenvalues $\lambda_{1}$, $\lambda_{2}$, $\lambda_{3}$ 
and $\lambda_{4}$ are negative. It is clearly seen that 
$\lambda_{1}=-r<0$, $\lambda_{4}=-\eta<0$, since \(r>0\) and \(\eta>0\). 
So, for the stability existence of \(E_{1}\), we should have 
\(\lambda_{2}<0\) and \(\lambda_{3}<0\), that is,
\begin{equation*}
\begin{aligned}
m_{{1}}&\alpha\,K-{\frac {\lambda\,\omega}{\eta}}
-d-{\frac {\gamma\,\omega}{\eta+\omega}}<0\quad 
\mbox{and}\quad 
{\frac {m_{{2}}\phi\,\alpha\,K}{a+K}}-d-\delta-{\frac {\gamma\,\omega}{\eta+\omega}}<0\\
&\Rightarrow m_{{1}}\alpha\,K<{\frac {\lambda\,\omega}{\eta}}+d+{\frac {\gamma\,\omega}{\eta+\omega}}
\quad\mbox{and}\quad 
{\frac {m_{{2}}\phi\,\alpha\,K}{a+K}}<d+\delta+{\frac {\gamma\,\omega}{\eta+\omega}}\\
&\Rightarrow m_{{1}}\alpha\,K<{\frac{\lambda\,\omega\left(\eta+\omega\right)
+\eta\,\gamma\,\omega+d\eta\left(\eta+\omega\right)}{\eta\,\left(\eta+\omega\right)}}
\quad \mbox{and}\quad 
{\frac {m_{{2}}\phi\,\alpha\,K}{a+K}}<{\frac {\left(d+\delta\right)
\,\left(\eta+\omega\right)+\gamma\,\omega}{\eta+\omega}}\\
&\Rightarrow m_{{1}}\alpha\,K\,\eta\,\left(\eta+\omega\right)
< \lambda\,\omega\left(\eta+\omega\right)+\eta\,\gamma\,\omega
+d\eta\left(\eta+\omega\right)
\quad \mbox{and} \quad  
m_{{2}}\phi\,\alpha\,K\,(\eta+\omega)< (a+K)
\,\left(d+\delta\right)\,\left(\eta+\omega\right)+\left(a+K\right)\gamma\,\omega\\
&\Rightarrow {\frac{m_{{1}}\alpha\,K\,\eta\,\left(\eta+\omega\right)}{\lambda\,\omega\left(\eta
+\omega\right)+\eta\,\gamma\,\omega+d\eta\left(\eta+\omega\right)}}<1 
\quad \mbox{and}\quad
{\frac{m_{{2}}\phi\,\alpha\,K\,(\eta+\omega)}{(a+K)\,\left(d+\delta\right)\,
\left(\eta+\omega\right)+\left(a+K\right)\gamma\,\omega}}<1
\end{aligned}
\end{equation*}
or, equivalently, $R_{0}<1$ and $R_{1}<1$ with
the critical parameters $R_{0}$ and $R_{1}$ given by
\eqref{eq:R0:R1}. 
\end{proof}

The conditions for stability of the pest free equilibrium point 
$E_1$ indicate that if the attack rate of the pest population \(\alpha\) 
is low, then the system may stabilize to the pest free steady state.

\begin{theorem}[stability of the healthy pest free equilibrium]
The healthy pest free equilibrium \(E_3=\left(\bar{X},0,\bar{I},\bar{A}\right)\) 
is locally asymptotically stable if, and only if,  
\begin{enumerate}
\item[(i)] $\bar{X}< \displaystyle {\frac{\lambda\,\bar{A}^{2}
+\left(\lambda+d+\gamma\right)\bar{A}+d}{m_{1}\,\alpha\,(1+\bar{A})}}$,

\item[(ii)] $C_{i}>0$, \(i=1,3\), 

\item[(iii)] \(C_{1}\,C_{2}-C_{3}>0\), 
\end{enumerate}
where
\begin{equation}
\label{eq:Cis}
\begin{split}
C_1&=-F_{{11}}-F_{{33}}+\eta,\\
C_2&=\left( F_{{11}}-\eta \right) F_{{33}}-F_{{11}}\eta
+{\frac { \left( \gamma\,\sigma\, \left( {X}^{3}+3\,\bar{X}^{2}a
+3\,a \right) +m_{{2}}{\phi}^{2}{\alpha}^{2} \left( 1
+\bar{A} \right) ^{2} \right) \bar{I}}{ \left( \bar{X}^{3}+a\right)  
\left( 1+\bar{A} \right) ^{2}}},\\
C_3&=\eta\,F_{{11}}F_{{33}}+{\frac {\bar{I}\,\bar{X}a{\alpha}^{2}\eta\,{\phi}^{2}m_{{2}}
\left( 1+\bar{A} \right) ^{2}-\sigma\,I\gamma\, \left( a
+\bar{X} \right) ^{3}F_{{11}}}{ \left( \bar{X}^{3}+a \right) \left( 1+\bar{A} \right)^{2}}},
\end{split}
\end{equation}
with
\begin{equation}
\label{eq:Fii}
\begin{split}
F_{{11}}&=r \left( 1-{\frac {2\,\bar{X}}{K}} \right) 
-{\frac {\phi\,\alpha\,a\,\bar{I}}{\left( a+\bar{X} \right) ^{2}}},\\
F_{{22}}&=m_{{1}}\alpha\,\bar{X}-\lambda\,\bar{A}-d-{\frac {\gamma\,\bar{A}}{1+\bar{A}}},\\
F_{{33}}&={\frac {m_{{2}}\phi\,\alpha\,\bar{X}}{a+\bar{X}}}
-d-\delta-{\frac {\gamma\,\bar{A}}{1+\bar{A}}}.
\end{split}
\end{equation}
\end{theorem}

\begin{proof}
At the healthy pest free fixed point 
\(E_3 = \left(\bar{X},0,\bar{I},\bar{A}\right)\), 
the Jacobian matrix is given by
\[
J(E_3)=
\left[ 
\begin{array}{cccc} 
F_{{11}}&-\alpha\,\bar{X}&-{\frac {\phi\,\alpha
\,\bar{X}}{a+\bar{X}}}&0\\ \noalign{\medskip}0&F_{{22}}&0&0\\ 
\noalign{\medskip}{\frac {m_{{2}}\phi\,\alpha\,a\,\bar{I}}{ 
\left( a+\bar{X} \right) ^{2}}}&\lambda\,\bar{A}&
F_{{33}}&-{\frac {\gamma\,\bar{I}}{ \left( 1+\bar{A} \right) ^{2}}}\\ 
\noalign{\medskip}0&\sigma&\sigma&-\eta
\end{array} 
\right]
\]
with $F_{ii}$, $i = 1, 2, 3$, as in \eqref{eq:Fii}.
The characteristic equation in $\rho$ is then given by
\begin{equation}
\label{charac}
\left( \rho-F_{{22}} \right)\left[{\rho}^{3}+C_{{1}}\,{\rho}^{2}+ C_2 \rho+C_3\right]=0,
\end{equation}
where the $C_{{i}}$, $i=1, 2, 3$, are defined by \eqref{eq:Cis}.
The equilibrium \(E_{3}\) is locally asymptotically stable 
if and only if all roots of the
polynomial \eqref{charac} have negative real parts. 
Equation \eqref{charac} has one root \(\rho=F_{22}\) 
and the other three roots are solution of
\begin{equation}
\label{eq:cub:pl3}
{\rho}^{3}+C_{{1}}\,{\rho}^{2}+ C_2 \rho+C_3=0. 
\end{equation}
To conclude about the stability behavior of \(E_{3}\), 
we  analyze the (three) roots of the cubic polynomial \eqref{eq:cub:pl3}.
The Routh--Hurwitz criteria applied to the third degree polynomial \eqref{eq:cub:pl3} 
tell us that a necessary and sufficient condition for the local stability of the system 
is that all eigenvalues must have negative real part, that is,  
$C_1>0$, $C_2>0$, $C_{3}>0$, and $C_{1}\,C_{2}-C_{3}>0$ must hold.
Hence, $E_3$ is locally asymptotically stable if, and only if, 
the following conditions hold:
\begin{itemize}
\item[(i)] \(F_{22}=m_{1}\,\alpha\,\bar{X}-\lambda\,\bar{A}
-d-\frac{\gamma A}{1+A}<0\Rightarrow\bar{X}<{\frac{\lambda\,\bar{A}^{2}
+\left(\lambda+d+\gamma\right)\bar{A}+d}{m_{1}\,\alpha\,(1+\bar{A})}}\);
\item[(ii)] \( C_{1}>0,C_{3}>0\);
\item[(iii)] \(C_{1}\,C_{2}-C_{3}>0\).
\end{itemize}
The proof is complete.
\end{proof}

\begin{theorem}[stability of the interior equilibrium point]
System \eqref{eq1} at the interior equilibrium point  
\(E^{*}=(X^{*},S^{*},I^{*},A^{*})\) 
is locally asymptotically stable if, and only if,
\begin{equation}
\label{cond:stab:E*}
\begin{gathered}
y_4 > 0,\\
y_{1}\,y_{2}-y_{3} > 0,\\
y_1\,y_2\,y_3-{y_3}^{{2}}-{y_{{1}}}^{{2}}\,y_{4} > 0,
\end{gathered}
\end{equation}
where
\begin{equation}
\begin{aligned}
\label{eq:def:yis}
y_1&=-\left(F_{{11}}+F_{{22}}+F_{{33}}\right)+\eta,\\
y_2&=\left( F_{{22}}+F_{{33}}-\eta \right) F_{{11}}+ \left( F_{{33}}-\eta
\right) F_{{22}}-\eta\,F_{{33}}+S^{*}X^{*}{\alpha}^{2}m_{{1}}
-{\frac {\gamma
\, \left( I^{*}+S^{*} \right) \sigma}{ \left( 1+A^{*} \right) ^{2} \left( X^{*}+a
\right)}}-{\frac {m_{{2}}{\phi}^{2}{\alpha}^{2}aI^{*}X^{*}}{ \left( X^{*}+a
\right)^{3}}},\\
y_3&=\left(  \left( -F_{{33}}+\eta \right) F_{{22}}+\eta\,F_{{33}}+{\frac 
{\gamma\, \left( I^{*}+S^{*} \right) \sigma}{ \left( 1+A^{*} \right) ^{2}}}
\right) F_{{11}}
+ \left( \eta\,F_{{33}}+ \left( -\lambda\,S^{*}+{\frac {\gamma\,I^{*}}{ \left( 1
+A^{*} \right) ^{2}}} \right) \sigma+{\frac {m_{{2}}{\phi}^{2}{\alpha}^{2}aI^{*}X^{*}}{ 
\left( X^{*}+a \right) ^{3}}} \right) F_{{22}}\\
&+\left( -S^{*}X^{*}{\alpha}^{2}m_{{1}}+\sigma\, \left( \lambda\,S^{*}+{\frac {
\gamma\,S^{*}}{ \left( 1+A^{*} \right) ^{2}}} \right)  \right) F_{{33}}
+S{\alpha}^{2} \left( \eta+{\frac {A^{*}\lambda\,\phi}{X^{*}+a}} \right) X^{*}m_{{1}}
-A^{*} \left( \lambda\,S^{*}+{\frac {\gamma\,S^{*}}{ \left( 1+A^{*} \right) ^{2}}}
\right) \sigma\,\lambda-{\frac {m_{{2}}{\phi}^{2}{\alpha}^{2}aI^{*}X^{*}\eta}{ 
\left( X^{*}+a \right) ^{3}}},\\
y_4&= \left(  \left( \lambda\,S^{*}-{\frac {\gamma\,I^{*}}{ \left( 1+A^{*}
\right) ^{2}}} \right) \sigma-\eta\,F_{{33}} \right) F_{{22}}-\sigma
\, \left( \lambda\,S^{*}+{\frac {\gamma\,S^{*}}{ \left( 1+A^{*} \right) ^{2}}}
\right) F_{{33}}+A^{*} \left(\lambda\,S^{*}+{\frac {\gamma\,
S^{*}}{ \left( 1+A^{*}\right) ^{2}}} \right) \sigma\,\lambda  F_{{11}}\\
&+{\frac {{\alpha}^{2}{\phi}^{2}\eta\,X^{*}F_{{22}}m_{{2}}aI}{ \left( X^{*}+a
\right) ^{3}}}-S^{*}X^{*}{\alpha}^{2}\eta\,m_{{1}}F_{{33}}-{\frac { \left( X^{*}+
a-\phi \right) m_{{2}}I^{*}\sigma\,a{\alpha}^{2} \left( \lambda\, \left( 1
+A^{*} \right) ^{2}+\gamma \right) \phi\,X^{*}S}{ \left( 1+A^{*} \right) ^{2}
\left( X^{*}+a \right) ^{3}}}\\
&+{\frac { \left(  \left( S^{*} \left( X^{*}+a-\phi \right) \sigma+A^{*}\eta\,\phi
\right)  \left( 1+A^{*} \right) ^{2}\lambda-I^{*}\gamma\,\sigma\, \left( X^{*}+a-
\phi \right)  \right) S^{*}X^{*}{\alpha}^{2}m_{{1}}}{ \left( 1+A^{*} \right) ^{2}
\left( X^{*}+a \right)}},	
\end{aligned}
\end{equation}
with
\begin{equation}
\label{eq:F11:F22:F33}
\begin{aligned}
F_{{11}}&=r \left( 1-{\frac {2\,{X}^{*}}{K}} \right) -\alpha\,\bar{S}-{\frac {\phi\,\alpha
\,a\,{I}^{*}}{ \left( a+{X}^{*} \right) ^{2}}},\\
F_{{22}}&=m_{{1}}\alpha\,{X}^{*}-\lambda\,{A}^{*}-d-{\frac {\gamma\,{A}^{*}}{1+{A}^{*}}},\\
F_{{33}}&={\frac {m_{{2}}\phi\,\alpha\,{X}^{*}}{a+{X}^{*}}}
-d-\delta-{\frac {\gamma\,{A}^{*}}{1+{A}^{*}}}.
\end{aligned}
\end{equation}
\end{theorem}

\begin{proof}
The Jacobian matrix at the coexistence equilibrium point $E^*$ is given by
\[
J(E^{*})
=\left[ 
\begin{array}{cccc} 
F_{{11}}&-\alpha\,X^{*}&-{\frac {\phi\,\alpha
\,X^{*}}{a+X^{*}}}&0\\ \noalign{\medskip}m_{{1}}\alpha\,S^{*}&F_{{22}}&0&
-\lambda\,S^{*}-{\frac {\gamma\,S^{*}}{ \left( 1+A^{*} \right) ^{2}}}\\ 
\noalign{\medskip}
{\frac {m_{{2}}\phi\,\alpha\,a\,I^{*}}{ \left( a+X^{*} \right) ^{2}}}&\lambda\,A^{*}
&F_{{33}}&\lambda\,S^{*}-{\frac {\gamma {I}^{*}}{ \left( 1+A^{*} \right) ^{2}}}\\ 
\noalign{\medskip}0&\sigma&\sigma&-\eta\end {array} \right], 
\]
where $F_{{ii}}$, $i = 1, 2, 3$, are given by \eqref{eq:F11:F22:F33}.
The characteristic equation in $\rho$ for the Jacobian matrix $J(E^*)$  is given by
\[
|\rho I-J(E^{*})|=
\left|
\begin {array}{cccc} \rho-F_{{11}}&\alpha\,X^{*}&{\frac {\phi\,\alpha\,X^{*}}{a+X^{*}}}&0\\ 
\noalign{\medskip}-m_{{1}}\alpha\,S^{*}&\rho-F_{{22}}&0&\lambda
\,S^{*}+{\frac {\gamma\,S^{*}}{ \left( 1+A^{*} \right) ^{2}}}\\ \noalign{\medskip}
-{\frac {m_{{2}}\phi\,\alpha\,a\,I^{*}}{ \left( a+X^{*} \right) ^{2}}}&-\lambda\,A^{*}
&\rho-F_{{33}}&-\lambda\,S^{*}+{\frac {\gamma\,I^{*}}{ \left( 1+A^{*} \right) ^{2}}}\\ 
\noalign{\medskip}0&-\sigma&-\sigma&\rho+\eta\end {array}
\right|=0,
\]
which gives
\begin{equation}
\label{char}
\rho^{{4}}+y_{1}\rho^{3}+y_{2}\rho^2+y_{3}\rho+y_{4}=0
\end{equation}
with $y_i$, $i = 1, \ldots, 4$, defined by \eqref{eq:def:yis}.
Then, noting that \(y_{1}>0\), it follows from the Routh--Hurwitz criterion 
that the interior equilibrium \(E^{*}\) is locally asymptotically stable if 
\eqref{cond:stab:E*} hold and unstable otherwise.
\end{proof}

Next, we shall find out conditions for which the system
enters into Hopf bifurcation around the interior equilibrium \(E^{*}\). 
We focus on the pest consumption rate $\alpha$, which is considered 
as the most biologically significant parameter. 


\section{Hopf-bifurcation}
\label{sec6}

Let us define the continuously differentiable function 
$\Psi:(0, \infty)\rightarrow \mathbb{R}$ of $\alpha$ as follows:
\begin{equation}
\label{eq:Psi}
\Psi(\alpha):=y_{1}(\alpha)\,y_{2}(\alpha)\,y_{3}(\alpha)
-{y_3}^{2}(\alpha)-{y_4}(\alpha){y_1}^{2}(\alpha),
\end{equation}
where we look to expressions \eqref{eq:def:yis} as functions of $\alpha$.
The conditions for occurrence of a Hopf-bifurcation tell us that the spectrum 
\(\sigma(\alpha)\) of the characteristic equation
should satisfy the following conditions:
\begin{itemize}
\item[(i)] there exists \(\alpha^{*}\in (0,\infty)\) at which a pair 
of complex eigenvalues \(\rho(\alpha^{*}),
{\bar{\rho}}(\alpha^{*})\in \sigma(\alpha)\) are such that
\[
\Re e{\rho}(\alpha^{*})=0,
\quad Im\rho(\alpha^{*})=\omega_{0}>0
\]
with transversality condition
\[
\left.\frac{d\Re e[\rho(\alpha)]}{d\alpha}\right|_{\alpha^{*}}\neq 0;
\]
\item[(ii)] all other elements of \(\sigma(\alpha)\) have negative real parts. 
\end{itemize}

We obtain the following result.

\begin{theorem}[Hopf bifurcation around the interior equilibrium 
with respect to the pest consumption rate $\alpha$]
\label{thm4.5}
Let $\Psi(\alpha)$ be given as in \eqref{eq:Psi}
and let \(\alpha^{*}\in (0, \infty)\) be such that 
$\Psi(\alpha^{*})=0$.
System \eqref{eq1} enters into a Hopf bifurcation 
around the coexistence equilibrium $E^{*}$  
at \(\alpha^{*}\) if and only if
\(A(\alpha^*)C(\alpha^*)+B(\alpha^*)D(\alpha^*)\neq0\),
where
\begin{equation}
\label{eq:ABCD}
\begin{split}
A(\alpha)&=4{\beta_{1}}^{3}-12\beta_{1}{\beta_{2}}^{2}
+3y_1\left({\beta_{1}}^{2}-{\beta_{2}}^{2}\right)+2y_2\beta_{1}+y_3,\\
B(\alpha)&=12{\beta_{1}}^{2}\,\beta_{2}+6y_1\beta_{1}\beta_{2}-4\,{\beta_{2}}^3+2y_{2}\beta_{2},\\
C(\alpha)&=\left( {\beta_{{1}}}^{3}-3\,\beta_{{1}}{\beta_{{2}}}^{2} \right) y_{{1}}'
+ \left( {\beta_{{1}}}^{2}-{\beta_{{2}}}^{2} \right) y_{{2}}'
+\beta_{{1}}y_{{3}}'+y_{{4}}',\\
D(\alpha)&=\left( 3\,{\beta_{{1}}}^{2}\beta_{{2}}-{\beta_{{2}}}^{3} \right) y_{{1}}'
+2\,\beta_{{1}}\beta_{{2}}y_{{2}}'+\beta_{{2}}y_{{3}}',
\end{split}
\end{equation}
with $\rho_1=\beta_{1}+i\beta_{2}$
and $\rho_2=\beta_{1}-i\beta_{2}$
the pair of conjugate complex eigenvalues,
solutions of the characteristic equation \eqref{char}, 
with \(\rho_i(\alpha)\) purely imaginary at  \(\alpha=\alpha^{*}\), $i = 1, 2$,
and where the other eigenvalues $\rho_3$ and $\rho_4$, 
solutions of \eqref{char}, have negative real parts.
\end{theorem}

\begin{proof}
The critical value \(\alpha^{*}\) is obtained from the equation \(\Psi(\alpha)=0\). 
For \(\alpha=\alpha^{*}\), we have 
\begin{equation*}
\begin{split}
\Psi(\alpha)=0
&\Rightarrow y_{1}(\alpha)\,y_{2}(\alpha)\,y_{3}(\alpha)
-{y_{3}}^{2}(\alpha)-y_{4}(\alpha)\,y_{1}^{2}(\alpha)=0\\
&\Rightarrow y_{1}(\alpha)\,y_{2}(\alpha)\,y_{3}(\alpha)
={y_{3}}^{2}(\alpha)+y_{4}(\alpha)\,y_{1}^{2}(\alpha)=0\\
&\Rightarrow y_{1}\,y_{2}\,y_{3}={y_{3}}^{2}+y_{4}\,{y_{1}}^{2},
\end{split} 
\end{equation*}
from which we get 
\[
y_{2}=\frac{y_{3}}{y_{1}}+\frac{y_{4}\,{y_{1}}}{y_{3}}.
\]
Then, the characteristic equation \eqref{char} becomes
\begin{equation*}
\begin{split}
\rho^{4}+y_{1}\,\rho^{3}+\left(\frac{y_{3}}{y_{1}}
+\frac{y_{1}\,y_{4}}{y_{3}}\right)\rho^{2}+y_{3}\,\rho+y_{4}=0
&\Rightarrow \rho^{4}+y_{1}\rho^{3}+\frac{y_{1}\,y_{4}}{y_{3}}\,\rho^2
+\frac{y_{3}}{y_{1}}\,\rho^{2}+y_{3}\,\rho+y_{4}=0\\
&\Rightarrow \rho^{2}\left(\rho^{2}+y_{1}\,\rho+\frac{y_{1}\,y_{4}}{y_{3}}\right)
+\frac{y_{3}}{y_{1}}\left(\rho^2+y_{1}\rho+\frac{y_{1}\,y_{4}}{y_{3}}\right)=0,
\end{split} 
\end{equation*}
that is,
\begin{equation}
\label{hopf}
\left(\rho^{2}+\frac{y_{3}}{y_{1}}\right)\left(\rho^2+y_{1}\rho
+\frac{y_{1}\,y_{4}}{y_{3}}\right)=0.
\end{equation}
We suppose equation \eqref{hopf} has four roots $\rho_i$, $i=1,2,3,4$, 
with the pair of purely imaginary roots $\rho_1$ and $\rho_2$ 
at $\alpha=\alpha^*$: $\rho_1=\bar{\rho_2}$. We get
\begin{equation}
\label{sim}
\begin{split}
\rho_3+\rho_4&=-y_1,\\
\omega_0^2+\rho_3\rho_4&=y_2,\\
\omega_0^2(\rho_3+\rho_4)&=-y_3,\\
\omega_0^2\rho_3\rho_4&=y_4,
\end{split}
\end{equation}
where $\omega_0=Im\rho_{1}(\alpha^*)$. From these relations, 
we obtain that \({\omega_0}^{2}={\frac{y_3}{y_1}}\).
Now, if $\rho_3$ and $\rho_4$ are complex conjugate, then 
from the first part of \eqref{sim}, it follows that \(2\,\Re e\rho_3=-y_1\). 
If they are real roots, then by the first and last parts of \eqref{sim}, 
$\rho_3<0$ and $\rho_4<0$.  Further, as $\psi(\alpha^*)$ is a continuous 
function of all its roots, there exists an open interval 
\(\alpha\in\left(\alpha^{*}-\epsilon,\alpha^{*}+\epsilon\right)\)
such that $\rho_1$ and $\rho_2$ are complex conjugate for $\alpha$.
Suppose their general forms in this neighborhood are
\begin{equation*}
\begin{aligned}
\rho_1(\alpha)&=\beta_{1}(\alpha)+i\beta_{2}(\alpha),\\
\rho_2(\alpha)&=\beta_{1}(\alpha)-i\beta_{2}(\alpha).
\end{aligned}
\end{equation*}
Now, we verify the transversality condition
\begin{equation}
\label{transversality}
\left.{\frac{d\Re e[\rho_j(\alpha)]}{d\alpha}}\right|_{\alpha=\alpha^*}\neq 0,
\quad j=1,2.
\end{equation}
Substituting \(\rho_j(\alpha)=\beta_{1}(\alpha)\pm i\beta_{2}(\alpha)\)
into \eqref{char}, we get the following equation:
\[
\left(\beta_{1}(\alpha)+i\,\beta_{2}(\alpha)\right)^{4}
+y_{1}\left(\beta_{1}(\alpha)+i\,\beta_{2}(\alpha)\right)^{3}
+y_{2}\left(\beta_{1}(\alpha)+i\,\beta_{2}(\alpha)\right)^{2}
+y_{3}\left(\beta_{1}(\alpha)+i\,\beta_{2}(\alpha)\right)+y_{4}=0.
\]
Differentiating with respect to \(\alpha\), we have
\begin{multline*}
4\left(\beta_{1}(\alpha)+i\,\beta_{2}(\alpha)\right)^{3}\left({\beta_{1}}'(\alpha)
+i\,{\beta_{2}}'(\alpha)\right)+{y_{1}}'\left(\beta_{1}(\alpha)+i\,\beta_{2}(\alpha)\right)^{3}
+3y_{1}\left(\beta_{1}(\alpha)+i\,\beta_{2}(\alpha)\right)^{2}\left({\beta_{1}}'(\alpha)
+i\,{\beta_{2}}'(\alpha)\right)\\
+{y_{2}}'\left(\beta_{1}(\alpha)+i\beta_{2}(\alpha)\right)^{2}
+2y_{2}\left({\beta_{1}}(\alpha)+i\,{\beta_{2}}(\alpha)\right)\left({\beta_{1}}'(\alpha)
+i\,{\beta_{2}}'(\alpha)\right)
+{y_{3}}'\left({\beta_{1}}(\alpha)+i\,{\beta_{2}}(\alpha)\right)
+y_{3}\left({\beta_{1}}'(\alpha)+i\,{\beta_{2}}'(\alpha)\right)+{y_{4}}'=0,
\end{multline*}
that is,
\begin{align*}
\underbrace{4\,i{\beta_{{1}}}^{3}{\beta_{{2}}}'+12\,i{\beta_{{1}}}^{2}\beta_{{2}}\beta_{{1}}'
+3\,i{\beta_{{1}}}^{2}\beta_{{2}}y_{{1}}'+3\,i{\beta_{{1}}}^{2}\beta_{{2}}'y_{{1}}
-12\,i\beta_{{1}}{\beta_{{2}}}^{2}\beta_{{2}}'+6\,i\beta_{{1}}\beta_{{2}}\beta_{{1}}'y_{{1}}}\\
\underbrace{-4\,i{\beta_{{2}}}^{3}\beta_{{1}}'-i{\beta_{{2}}}^{3}y_{{1}}'
-3\,i{\beta_{{2}}}^{2}\beta_{{2}}'y_{{1}}+2\,i\beta_{{1}}\beta_{{2}}y_{{2}}'
+2\,i\beta_{{1}}\beta_{{2}}'y_{{2}}+2\,i\beta_{{2}}\beta_{{1}}'y_{{2}}}\\
\overbrace{+4\,{\beta_{{1}}}^{3}\beta_{{1}}'+{\beta_{{1}}}^{3}y_{{1}}'
-12\,{\beta_{{1}}}^{2}\beta_{{2}}\beta_{{2}}'
+3\,{\beta_{{1}}}^{2}\beta_{{1}}'y_{{1}}-12\,\beta_{{1}}{
\beta_{{2}}}^{2}\beta_{{1}}'-3\,\beta_{{1}}{\beta_{{2}}}^{2}y_{{1}}'
-6\,\beta_{{1}}\beta_{{2}}\beta_{{2}}'y_{{1}}+4\,{\beta_{{2}}}^{3}\beta_{{2}}'}\\
\overbrace{-3\,{\beta_{{2}}}^{2}\beta_{{1}}'y_{{1}}+i\beta_{{2}}y_{{3}}'
+i\beta_{{2}}'y_{{3}}+{\beta_{{1}}}^{2}y_{{2}}'+2\,\beta_{{1}}\beta_{{1}}'y_{{2}}
-{\beta_{{2}}}^{2}y_{{2}}'-2\,\beta_{{2}}\beta_{{2}}'y_{{2}}
+\beta_{{1}}y_{{3}}'+\beta_{{1}}'y_{{3}}+y_{{4}}'}=0.
\end{align*}
Comparing the real and imaginary parts, we get
\begin{multline}
\label{real}
\left( 4\,{\beta_{{1}}}^{3}-12\,\beta_{{1}}{\beta_{{2}}}^{2}+3\,y_{{1}}
\left( {\beta_{{1}}}^{2}-{\beta_{{2}}}^{2} \right) +2\,\beta_{{1}}y_{
{2}}+y_{{3}} \right) \beta_{{1}}'
+\left( -12\,{\beta_{{1}}}^{2}\beta_{{2}}-6\,\beta_{{1}}\beta_{{2}}y_{
{1}}+4\,{\beta_{{2}}}^{3}-2\,\beta_{{2}}y_{{2}} \right) \beta_{{2}}'\\
+\left( {\beta_{{1}}}^{3}-3\,\beta_{{1}}{\beta_{{2}}}^{2} \right) y_{{
1}}'+ \left( {\beta_{{1}}}^{2}-{\beta_{{2}}}^{2} \right) y_{{2}}'+
\beta_{{1}}y_{{3}}'+y_{{4}}'=0
\end{multline}
and
\begin{multline}
\label{imag}
\left(12\,{\beta_{{1}}}^{2}\beta_{{2}}+6\,\beta_{{1}}\beta_{{2}}y_{{1}}
-4\,{\beta_{{2}}}^{3}+2\,\beta_{{2}}y_{{2}} \right) \beta_{{1}}'
+\left( 4\,{\beta_{{1}}}^{3}-12\,\beta_{{1}}{\beta_{{2}}}^{2}+3\,y_{{1
}} \left( {\beta_{{1}}}^{2}-{\beta_{{2}}}^{2} \right) +2\,\beta_{{1}}y
_{{2}}+y_{{3}} \right) \beta_{{2}}'\\
+\left( 3\,{\beta_{{1}}}^{2}\beta_{{2}}-{\beta_{{2}}}^{3} \right) y_{{1}}'
+2\,\beta_{{1}}\beta_{{2}}y_{{2}}'+\beta_{{2}}y_{{3}}'=0.		
\end{multline}
Equivalently, we can write \eqref{real} and \eqref{imag}, in a compact form, as
\begin{equation}
\label{compact}
\begin{split}
A(\alpha)\beta_{1}'(\alpha)-B(\alpha)\beta_{2}'(\alpha)+C(\alpha)&=0,\\
B(\alpha)\beta_{1}'(\alpha)+A(\alpha)\beta_{2}'(\alpha)+D(\alpha)&=0,
\end{split}
\end{equation}
where $A(\alpha)$, $B(\alpha)$, $C(\alpha)$ and $D(\alpha)$ are defined 
by \eqref{eq:ABCD}.
Hence, solving system \eqref{compact} for $\beta_{1}'(\alpha^{*})$, we get
\[
\left.{\frac{d\Re e[\rho_{j}(\alpha)]}{d\alpha}}\right|_{\alpha=\alpha^{*}}
= \left.{\beta_{1}'(\alpha)}\right|_{\alpha=\alpha^{*}}
=-\frac{A(\alpha^*)C(\alpha^*)+B(\alpha^*)D(\alpha^*)}{A^2(\alpha^*)+B^2(\alpha^*)}\neq0.
\]
Thus, the transversality conditions hold 
if and only if \(A(\alpha^*)C(\alpha^*)+B(\alpha^*)D(\alpha^*)\neq0\),
in which case a Hopf bifurcation occurs at $\alpha=\alpha^*$.
\end{proof} 

We have restricted ourselves here to study the Hopf bifurcation around 
the interior equilibrium point with respect to the pest consumption rate 
$\alpha$, because it is the most biologically significant parameter. However,
by replacing \(\alpha\) by other model parameters, such as $\lambda$, $\gamma$
or $\sigma$, one can also study the Hopf bifurcation around the interior equilibrium 
point with respect to such parameters of the model.


\section{The optimal control problem}
\label{sec7}

In this section, we introduce an optimal control problem that consists
to minimize the negative effects of chemical pesticides and to minimize the cost of pest management.
We extend the model system \eqref{eq1} by incorporating three time-dependent controls: 
$u_1(t)$, \(u_{2}(t)\) and $u_3(t)$, where the first control $u_1$ is for controlling 
the use of chemical pesticides, the second control $u_2$ is for bio-pesticides, 
and the third control $u_3$ is for advertisement.
The objective is to reduce the price of announcement for farming awareness via radio, TV, 
telephony and other social media, while taking into account the price regarding the control measures. 
Our target is to find optimal functions ${u_1}^{*}(t)$, ${u_{2}}^{*}(t)$ and ${u_3}^{*}(t)$ 
using the PMP \cite{Fleming}. In agreement, our system \eqref{eq1} is modified 
into the induced nonlinear dynamic control system given by
\begin{equation}
\label{controlsystem}
\begin{cases}
\frac{dX}{dt} =r X\left(1 - \frac{X}{K}\right) - \alpha XS - \frac{\phi \alpha XI}{a+X} ,\\
\frac{dS}{dt} =m_{1}\alpha XS - u_2\lambda AS - dS- \frac{u_1\gamma SA}{1+A},\\
\frac{dI}{dt} =\frac{m_{2}\phi \alpha XI}{a+X} + u_2\lambda AS - (d+\delta)I- \frac{u_1\gamma IA}{1+A},\\
\frac{dA}{dt} =u_3\omega + \sigma(S+I) - \eta A,
\end{cases}
\end{equation}
with given initial conditions 
\begin{equation}
\label{initial}
X(0)=X_0, \quad S(0)=S_0, \quad I(0)=I_0 \quad \mbox{and}\quad A(0)=A_0.
\end{equation}
We need to reduce the number of pests and also the price of pest administration 
by reducing the cost of pesticides and exploiting the stage of awareness. 
The objective cost functional for the minimization problem is denoted 
by \(J(u_{1},u_{2},u_{3})\) and defined as follows:
\begin{equation}
\label{obj}
J(u_{1}(\cdot),u_{2}(\cdot),u_{3}(\cdot))
=\int_{0}^{t_{f}}g(t,\varPhi(t),u(t))dt
=\int^{t_f}_{0}\left[
\frac{P_1\,{u_{1}}^{2}(t)}{2}
+\frac{P_2\,{u_{2}}^{2}(t)}{2}
+\frac{P_{3}{u_{3}}^{2}(t)}{2}+Q\,S^{2}(t)-R\,{A^{2}(t)}
\right]dt,
\end{equation}
where \(\varPhi(t)=(X(t),S(t),I(t),A(t))\) is the solution 
to the induced control system \eqref{controlsystem}--\eqref{initial},
$t \in [0, t_f]$, for the specific control \(u(t) = (u_{1}(t),u_{2}(t),u_{3}(t))\); 
the amounts $\frac{P_1}{2}, \frac{P_2}{2}$ and $\frac{P_{3}}{2}$ 
are the positive weight constants on the benefit of the cost; and the terms $Q$ and $R$ 
are the penalty multipliers. We prefer a quadratic cost functional on the controls, 
as an approximation for the nonlinear function depending on the assumption that the cost 
takes a nonlinear form, and also to prevent the bang-bang or singular optimal control cases. 
The control set is defined on \([t_0, t_f]\) subject to the constraints 
$0 \leq u_{i}(t) \leq 1$, $i = 1, 2, 3$, where \(t_0 \) and \(t_f\) are the starting 
and final times of the optimal control problem, respectively. The aim is to find the optimal profile 
of \(u_{1}(t), u_{2}(t)\) and \(u_{3}(t)\), denoted by ${u_{i}}^{*}(t)$, $i = 1, 2, 3$, 
so that the cost functional \(J\) has a minimum value, that is,
\begin{equation}
\label{u}
J(u_1^*,u_2^*,u_3^*)
= \min(J\left(u_1,u_2,u_3\right)\colon\left(u_1,u_2,u_3\right)\in \mathcal{U})
\end{equation} 
subject to \eqref{controlsystem}--\eqref{initial}, where
\begin{equation}
\label{controlset}
\mathcal{U}=\left\{u=\left(u_{1},u_{2},u_{3}\right) \in L^1 \, | \, 
0 \leq u_{1}(t)\leq 1, \ 0\leq u_{2}(t)\leq 1,\ 
0 \leq u_{3}(t)\leq 1,\  t\in[0,t_{f}]\right\}
\end{equation}
is the admissible control set with $L^1$ the class of Lebesgue measurable functions.
The PMP \cite{Pontryagin} is used to find the optimal control triplet
\(u^{*}=(u_{1}^{*},u_{2}^{*},u_{3}^{*})\). For that the Hamiltonian function is defined as
\begin{equation}
\mathcal{H}=\left[\frac{P_{1}u_{1}^{2}}{2}
+\frac{P_{2}u_{2}^{2}}{2}\frac{+P_{3}u_{3}^{2}}{2}
+Q\,{S^{2}}-R\,A^{2}\right]+\sum_{i=1}^{4}
\lambda_{i}\, f_{i}(X,S,I,A),
\end{equation}
where $\lambda_{i}$, $i = 1, 2,\ldots, 4$, are the adjoint variables 
and $f_{i}$, $i=1,2,3,4$, are the right-hand sides of system 
\eqref{controlsystem} at the $i^{th}$ state. Before trying
to find the solution of the optimal control problem through
the PMP, one first needs to prove that the problem has a solution.


\subsection{Existence of solution}

The existence of an optimal control triple can be guaranteed by using well-known results 
\cite{Fleming}. Since all the state variables involved in the model are continuously
differentiable, existence of solution is guaranteed under the following conditions
\cite{Fleming,Pontryagin,Suzan}:
\begin{itemize}
\item[(i)] The set of trajectories to system \eqref{controlsystem}--\eqref{initial} 
on the admissible class of controls \eqref{controlset} is non-empty.

\item[(ii)] The set where the controls take values is convex and closed.

\item[(iii)] Each right hand side of the state system \eqref{controlsystem} 
is continuous, is bounded above by a sum of the bounded control and
the state, and can be written as a linear function of \(u\)
with coefficients depending on time and the state variables.

\item[(iv)] The integrand \(g(t,\varPhi,u)\) of the objective 
functional \eqref{obj} is convex with respect to the control variables.

\item[(v)] There exist positive numbers 
\(\ell_{1},\ell_{2},\ell_{3},\ell_{4}\) and a constant \(\ell>1\) such that
\[
g(t,\varPhi,u)\geq-\ell_{1}+\ell_{2}|u_{1}|^{\ell}+\ell_{3}|u_{2}|^{\ell}+\ell_{4}|u_{3}|^{\ell}.
\] 
\end{itemize}	

We obtain the following existence result.

\begin{theorem}
Consider the optimal control problem defined by: the
objective functional \eqref{obj} on \eqref{controlset};
the control system \eqref{controlsystem}; and nonnegative 
initial conditions \eqref{initial}. Then there exists an optimal 
control triple \(u^* = (u_{1}^*,u_{2}^*,u_{3}^{*})\) 
and corresponding state trajectory 
\(\varPhi^{*} = \left(X^{*},S^{*},I^{*},A^{*}\right)\) 
such that \(J(u_{1}^*,u_{2}^*,u_{3}^*)
=\minof{\underset{\mathcal{U}}{\min}}J\left(u_{1},u_{2},u_{3}\right)\)
subject to \eqref{controlsystem}--\eqref{initial}. 
\end{theorem}

\begin{proof}
The proof is done verifying each of the five items (i)--(v) stated above.
\begin{itemize}
\item[(i)] Since \(\mathcal{U}\) is a nonempty set of real valued measurable functions 
on the finite time interval \(0\leq t \leq t_{f}\), the system \eqref{controlsystem} 
has bounded coefficients and hence any solutions are bounded on \([0,t_{f}]\)
(see Theorem~\ref{bounded}). It follows that the corresponding solutions for 
system \eqref{controlsystem}--\eqref{initial} exist \cite{Lukes}.

\item[(ii)] In our case, the set \(\Omega\) where the admissible controls take values 
is $\Omega=\left\{u\in\mathbb{R}^{3}\colon||u||\leq1\right\}$,
which is clearly a convex and closed set.

\item[(iii)] The right-hand sides of equations of system \eqref{controlsystem} 
are continuous. All variables \(X,S,I,A\) and \(u\) are bounded on \([0,t_{f}]\) 
and can be written as a linear function of \(u_{1}, u_{2},\) and \(u_{3}\) 
with coefficients depending on time and state variables.

\item [(iv)] The integrand \(g(t,\varPhi,u)\) of \eqref{obj} is 
quadratic with respect to the control variables, so it is trivially convex. 

\item [(v)] Finally, it remains to show that there exists 
a constant \(\ell^{*}>1\) and positive constants 
\(\ell_{1},\ell_{2},\ell_{3}\) and \(\ell_{4}\) such that
\[
\frac{P_1\,{u_{1}}^{2}}{2}+\frac{P_2\,{u_{2}}^{2}}{2}
+\frac{P_{3}{u_{3}}^{2}}{2}+Q\,S^{2}-R\,{A^{2}}
\geq-\ell_{1}+\ell_{2}|u_{1}|^{\ell}+\ell_{3}|u_{2}|^{\ell}+\ell_{4}|u_{3}|^{\ell}.
\]
In Section~\ref{sec2}, we already showed that the state variables are bounded. 
Let $\ell_{1}=\sup\left(Q\,S^{2}-R\,A^{2}\right)$, $\ell_{2}=P_{1}$, 
$\ell_{3}=P_{2}$, $\ell_{4}=P_{3}$ and $\ell=2$. It follows that
\[
\frac{P_1\,{u_{1}}^{2}}{2}+\frac{P_2\,{u_{2}}^{2}}{2}
+\frac{P_{3}{u_{3}}^{2}}{2}+Q\,S^{2}-R\,{A^{2}}
\geq-\ell_{1}+\ell_{2}|u_{1}|^{\ell}+\ell_{3}|u_{2}|^{\ell}+\ell_{4}|u_{3}|^{\ell}.
\]
\end{itemize}
We conclude that there exists an optimal control triple \cite{Fleming}.
\end{proof}


\subsection{Characterization of the solution}

Since we know that there exists an optimal control 
triple for minimizing the functional
\[
J(u_{1},u_{2},u_{3})=\int^{t_f}_{0}
\left[\frac{P_1\,{u_{1}}^{2}(t)}{2}+\frac{P_2\,{u_{2}}^{2}(t)}{2}
+\frac{P_{3}{u_{3}}^{2}(t)}{2}+Q\,S^{2}(t)-R\,{A^{2}}(t)\right]dt
\] 
subject to the controlled system \eqref{controlsystem}
and initial conditions \eqref{initial}, 
we now derive, using the PMP, necessary conditions 
to characterize and find the optimal control triple \cite{Fleming,Pontryagin}. 
The necessary conditions include: the minimality condition,
the adjoint system, and the transversality conditions, 
which come from the PMP \cite{Pontryagin}. Roughly speaking,
the PMP reduces the optimal control problem, a dynamic optimization
problem, into a static optimization problem that consists of minimizing 
point-wise the Hamiltonian function \(\mathcal{H}\). 
The Hamiltonian associated to our problem is explicitly given by
\begin{equation*}
\begin{aligned}
\mathcal{H}(t,\varPhi,u,\lambda)
&=\frac{P_1\,{u_{1}}^{2}}{2}+\frac{P_2\,{u_{2}}^{2}}{2}
+\frac{P_{3}{u_{3}}^{2}}{2}+Q\,S^{2}-R\,{A^{2}}\\
&+\lambda_{1}\,\left(r X\left(1 - \frac{X}{K}\right) 
- \alpha XS - \frac{\phi \alpha XI}{a+X}\right)\\
&+\lambda_{2}\,\left(m_{1}\alpha XS - u_2\lambda AS 
- dS- \frac{u_1\gamma SA}{1+A}\right)\\
&+\lambda_{3}\,\left(\frac{m_{2}\phi \alpha XI}{a+X} 
+ u_2\lambda AS - (d+\delta)I- \frac{u_1\gamma IA}{1+A}\right)\\
&+\lambda_{4}\,\left(u_3\omega + \sigma(S+I) - \eta A\right).
\end{aligned}
\end{equation*}
The PMP asserts that if the control \(u^{*}=(u_{1}^{*},{u_{2}}^{*},{u_{3}}^{*})\) 
and the corresponding state \(\varPhi^{*}=(X^{*},S^{*},I^{*},A^{*})\) 
form an optimal couple, then, necessarily, there exists a
non-trivial adjoint vector \(\lambda=(\lambda_{1},\lambda_{2},\lambda_{3},\lambda_{4})\) 
satisfying the following Hamiltonian system \cite{berhe}:
\begin{equation*}
\begin{cases}
\frac{d\varPhi}{dt}=\frac{\partial\mathcal{H}(t,\varPhi,u,\lambda)}{\partial\lambda},\\
\frac{d\lambda}{dt}=-\frac{\partial\mathcal{H}(t,\varPhi,u,\lambda)}{\partial\varPhi},\\
\end{cases}
\end{equation*}
subject to initial conditions \eqref{initial} and transversality conditions 
$\lambda(t_f) =0$. Moreover, at each point of time $t$,
the optimal controls are characterized by
\begin{align}
\begin{cases}
\,u_{i}^{*}=1, &\text{if $\frac{\partial\mathcal{H}}{\partial u_{i}}<0$,} \\
\, {u_{i}}^{*}=0, &\text{if $\frac{\partial\mathcal{H}}{\partial u_{i}}>0$.}
\end{cases}
\end{align}

\begin{theorem}
\label{thm:PMP:ourOCP}
If the controls  $({u_1}^{*},{u_2}^{*},{u_{3}}^{*})$ and the corresponding trajectories
\((H^{*},S^{*},I^{*},A^{*})\) are optimal, then there exist adjoint variables 
$\lambda_1$, $\lambda_2$, $\lambda_3$ and $\lambda_4$ satisfying 
the system of equations
\begin{equation}
\label{adjoint}
\begin{cases}
\frac{d\lambda_1}{dt}&=\lambda_1\,\left(\alpha S+\frac{\phi\alpha aI}{(a+X)^2}
-r\left(1-\frac{2X}{K}\right)\right)-\lambda_2\,m_1\,\alpha S
-\lambda_3\,\frac{m_2\,\phi\alpha\, aI}{(a+X)^2},\\
\frac{d\lambda_2}{dt}&=-2\,Q\,S+\lambda_1\,\alpha X
+\lambda_2\left(\frac{u_{1}\gamma A}{1+A}+u_{2}\,\lambda\,A+d
-m_1\alpha X\right)-\lambda_3\,u_2\lambda A-\lambda_4\sigma,\\
\frac{d\lambda_3}{dt}&=\lambda_1\frac{\phi\alpha X}{a+X}
+\lambda_3\left(\frac{u_{1}\,\gamma\,A}{1+A}+d+\delta
-\frac{m_2\phi\alpha X}{a+X}\right)-\lambda_4\sigma,\\
\frac{d\lambda_4}{dt}&=2RA+\lambda_2\,\left(\frac{u_1\gamma S}{(1+A)^2}
+u_{2}\,\lambda\,S\right)+\lambda_3\,\left(\frac{u_{1}\gamma I}{(1+A)^2}
-u_{2}\,\lambda\,S\right)+\lambda_4\eta,
\end{cases}
\end{equation}
with transversality conditions 
\begin{equation}
\label{trans}
\lambda_i(t_f)=0, \quad i=1,2,3,4.
\end{equation}
Furthermore, for \(t\in[0,t_{f}]\), the optimal controls 
$u_1^*,u_2^*$ and \(u_{3}^{*}\) are characterized by
\begin{equation}
\label{ui}
\begin{aligned}
{u_1}^{*}(t)&=\max\left\{0,\min\left\{1,\frac{(\lambda_{2}(t)\,S(t)
+\lambda_{3}(t)\,I(t))\gamma A(t)}{P_1(1+A(t))}\right\}\right\},\\
{u_2}^{*}(t)&=\max\left\{0,\min\left\{1,\frac{(\lambda_{2}(t)
-\lambda_{3}(t))\lambda\,A(t)S(t)}{P_2}\right\}\right\},\\
{u_3}^{*}(t)&=\max\left\{0,\min\left\{1,
-\frac{\lambda_4(t)\omega}{P_{3}}\right\}\right\}.
\end{aligned}
\end{equation}
\end{theorem}

\begin{proof}
The result is a direct consequence of the PMP.
\end{proof}

In order to confirm the nature of the Pontryagin extremals 
given by Theorem~\ref{thm:PMP:ourOCP}, we check the Hessian 
matrix of the Hamiltonian \(\mathcal{H}\). 
Because the Hessian matrix of \(\mathcal{H}\) with respect 
to the control variables is given by
\begin{equation*}
\frac{\partial^{2}\mathcal{H}}{\partial u^{2}}
=\left[
\begin{array}{ccc}
\frac{\partial^{2}\mathcal{H}}{\partial u_{1}^{2}}
&\frac{\partial^{2}\mathcal{H}}{\partial u_{1}\partial u_{2}}
&\frac{\partial^{2}\mathcal{H}}{\partial u_{1}\partial u_{3}}\\
\frac{\partial^{2}\mathcal{H}}{\partial u_{2}\partial u_{1}}
&\frac{\partial^{2}\mathcal{H}}{\partial u_{2}^{2}}
&\frac{\partial^{2}\mathcal{H}}{\partial u_{2}\partial u_{3}}\\
\frac{\partial^{2}\mathcal{H}}{\partial u_3\partial u_{1}}
&\frac{\partial^{2}\mathcal{H}}{\partial u_{3}\partial u_{2}}
&\frac{\partial^{2}\mathcal{H}}{\partial u_{3}^{2}}
\end{array}\right]
=\left[
\begin{array}{ccc}
P_{1}&0&0\\
0&P_{2}&0\\
0&0&P_{3}
\end{array}\right],
\end{equation*}
which is a positive definite matrix as a consequence 
of the positive weights \(P_{1},P_{2},P_{3}\), 
the Hamiltonian \(\mathcal{H}\) is convex
with respect to the control variables and, as a result, 
the Pontryagin extremals will be minimizers and not maximizers. 


\subsection{The method to solve the optimal control problem}

The optimal controls and the corresponding state functions 
are found by solving a system of dynamics called the 
\emph{optimality system}, and consisting of 
\begin{itemize}
\item[(i)] the state system \eqref{controlsystem} 
together with their initial conditions\eqref{initial};
\item[(ii)] the adjoint system \eqref{adjoint};
\item[(iii)] the terminal conditions \eqref{trans}; 
\item[(iv)] and the characterization of the optimal controls \eqref{ui}.
\end{itemize}
In combination, the method consists to solve the system
\begin{equation}
\label{optimality}
\begin{cases}
&\frac{dX}{dt}=r X\left(1 - \frac{X}{K}\right) 
- \alpha XS - \frac{\phi \alpha XI}{a+X} ,\\
&\frac{dS}{dt}=m_{1}\alpha XS - u_2\lambda AS - dS- \frac{u_1\gamma SA}{1+A},\\
&\frac{dI}{dt}=\frac{m_{2}\phi \alpha XI}{a+X} 
+ u_2\lambda AS - (d+\delta)I- \frac{u_1\gamma IA}{1+A},\\
&\frac{dA}{dt}=u_3\omega + \sigma(S+I) - \eta A,\\
&X(0) = X_0\geq 0, \ S(0) = S_0\geq 0, \ I(0) = I_0 \geq 0, \ A(0) = A_0 \geq 0,\\
&\frac{d\lambda_1}{dt}=\lambda_1\,\left(\alpha S+\frac{\phi\alpha aI}{(a+X)^2}
-r\left(1-\frac{2X}{K}\right)\right)-\lambda_2\,m_1\,\alpha S
-\lambda_3\,\frac{m_2\,\phi\alpha\, aI}{(a+X)^2},\\
&\frac{d\lambda_2}{dt}=-2\,Q\,S+\lambda_1\,\alpha X
+\lambda_2\left(\frac{u_{1}\gamma A}{(1+A)}+u_{2}\,\lambda\,A+d
-m_1\alpha X\right)-\lambda_3\,u_2\lambda A-\lambda_4\sigma,\\
&\frac{d\lambda_3}{dt}=\lambda_1\frac{\phi\alpha X}{a+X}
+\lambda_3\left(\frac{u_{1}\,\gamma\,A}{1+A}+d+\delta
-\frac{m_2\phi\alpha X}{a+X}\right)-\lambda_4\sigma,\\
&\frac{d\lambda_4}{dt}=2RA+\lambda_2\,\left(\frac{u_1\gamma S}{(1+A)^2}
+u_{2}\,\lambda\,S\right)+\lambda_3\,\left(\frac{u_{1}\gamma I}{(1+A)^2}
-u_{2}\,\lambda\,S\right)+\lambda_4\eta,\\
&\lambda_{1}(t_{f})=\lambda_{2}(t_{f})=\lambda_{3}(t_{f})=\lambda_{4}(t_{f})=0,
\end{cases}
\end{equation}
where $u_1$, $u_2$ and $u_3$ are given as in \eqref{ui}.
It is important to note that the adjoint system \eqref{adjoint} is also linear 
in $\lambda_{i}$ for \(i=1,2,3,4\) with bounded coefficients. Thus, there exists 
a positive real number \(M\) such that \(|\lambda_{i}|\leq M\) on \(t\in[0,t_{f}]\). 
Hence, for a sufficiently small time \(t_{f}\), the solution 
to the optimality system \eqref{optimality} is unique.
The need for a small time interval in order to guarantee uniqueness
of solution is due to the opposite
time direction/orientations of the optimality system: the state system 
has initial values while the adjoint system has terminal values.
Solving \eqref{optimality} analytically is not possible.
Consequently, we use a numerical method to find the approximate 
optimal solutions \(\varPhi^{*}\) and \(u^{*}\).

In Section~\ref{sec8}, we solve the optimal control problem numerically 
and observe the behavior of some solutions as time varies.


\section{Numerical simulations}
\label{sec8}

Since the analytical solution of system \eqref{eq1} is not practical 
to analyze, the numerical results play a great role in characterizing 
the dynamics. Our numerical simulations show how realistic our results 
are and illustrate well the predicted analytical behavior.
We begin by analyzing system \eqref{eq1} without controls,
then our control system \eqref{controlsystem} subject to the optimal controls, 
as characterized by the PMP. Our numerical simulations are acquired 
with a set of parameter values as given in Table~\ref{table1.1}. 
\begin{table}[ht!]
\begin{center}
\caption{Parameter values used in our numerical simulations.}
\begin{tabular}{|c|l|l|l||} \hline
Parameters  & Description  & Value &Source\\ \hline
$r$ & Growth rate of crop biomass & 0.05 per day &\cite{impact} \\
$K$ & Maximum density of crop biomass & 1 $m^{-2}$ &\cite{TAbraha}\\
$\lambda$ & Aware people activity rate& 0.025 per day&\cite{JTB}\\
$d$ & Natural mortality of pest& 0.01 day$^{-1}$&\cite{TAbraha}\\
$m_1$ & Conversion efficacy of susceptible pests & 0.8 &\cite{JTB}\\
$m_2$& Conversion efficiency of infected pest& 0.6& \cite{JTB}\\
$\delta$ & Disease related mortality rate & 0.1 per day &\cite{TAbraha}\\
$a$ & Half saturation constant &  0.2 & \cite{impact}\\
$\alpha$ & Attack rate of pest & 0.025 pest$^{-1}$per day&\cite{JTB}\\
$\sigma$ & Local rate of increase of awareness& 0.015 per day& Assumed\\
$\gamma$& The increase of level from global advertisement&0.025&\cite{EC19}\\
$\eta$ & Fading of memory of aware people& 0.015 {day$^{-1}$}& \cite{impact} \\
$\omega$ & Rate of global awareness (via TV, radio) & 0.003 day$^{-1}$ &Assumed\\ \hline
\end{tabular}
\label{table1.1}
\end{center}
\end{table}
For our numerical experiments of the uncontrolled system, 
we take $t_{f}=600$ days; while for the numerical simulations 
of the optimal control problem we fix $t_{f}=60$ days.  
The values of the weight function are taken as 
$P_1 =0.8$, $P_2=0.5$, $P_{3}=0.5$ $Q=10$, and $R=10$, 
and the initial state variables 
as $X(0)=0.2$, $S(0)=0.07$, $I(0)=0.05$, $A(0)=0.5$.
In Figures~\ref{fig1} and \ref{fig2}, the time series solution of model system 
\eqref{eq1} are sketched with different values of the parameters $\alpha$ and $\gamma$.
\begin{figure}[ht!]
\centering
\includegraphics[width=5in]{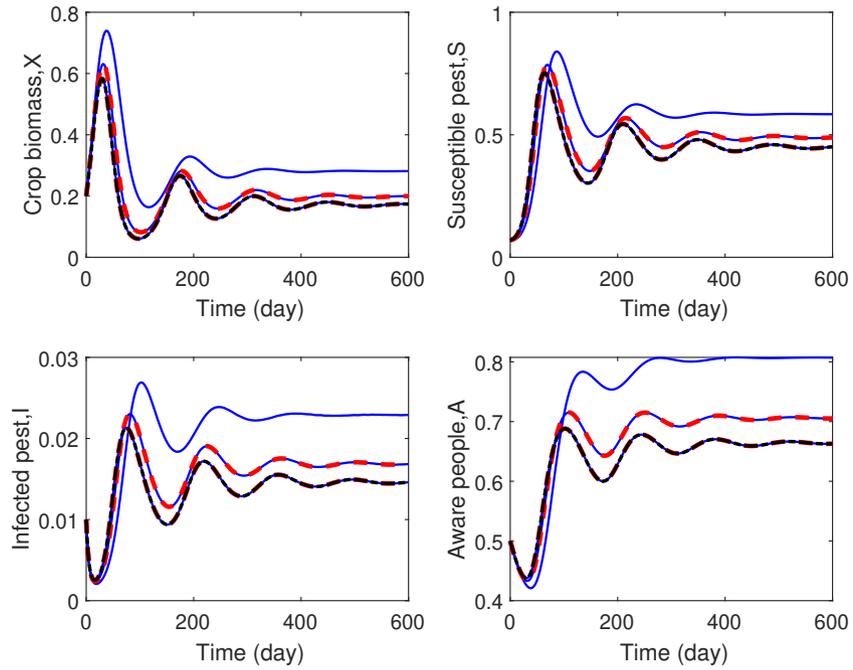}
\caption{Numerical solution of system \eqref{eq1} 
for different values of the rate $\alpha$ of pest: 
$\alpha=0.12$ (blue line), $\alpha=0.16$ (red line), 
$\alpha=0.18$ (black line). 
Other parameter values as in Table~\ref{table1.1}.}
\label{fig1}
\end{figure}
\begin{figure}[ht!]
\centering
\includegraphics[width=5in]{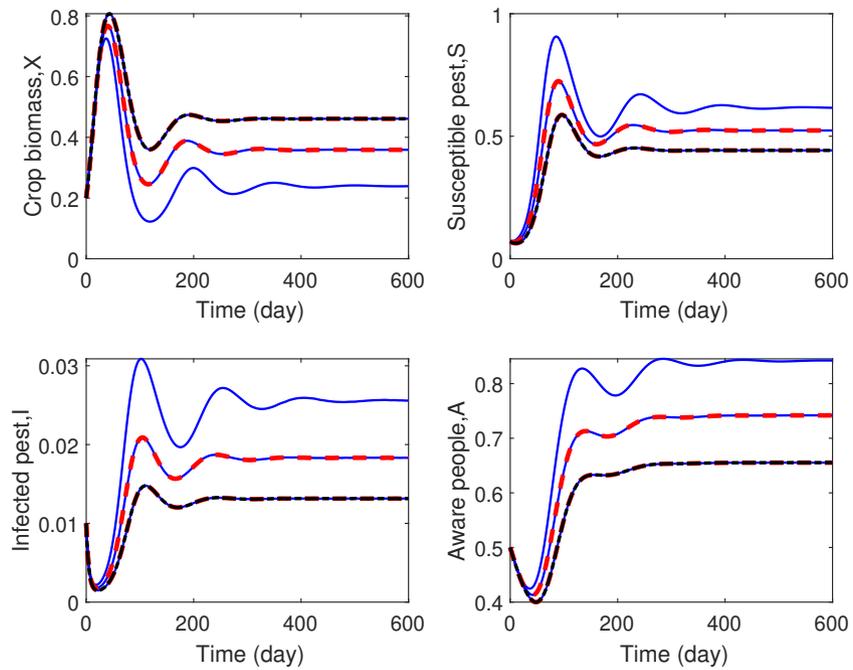}
\caption{Numerical solution of system \eqref{eq1} 
for different values of the rate $\gamma$ of pest: 
$\gamma=0.01$ (blue line), $\gamma=0.0.04$ (red line), 
$\gamma=0.07$ (black line). 
Other parameter values as in Table~\ref{table1.1}.}
\label{fig2}
\end{figure}
It is observed that our model variables $X(t)$, $S(t)$, $I(t)$ and $A(t)$ become oscillating
as the values of the rates (i.e., $\alpha$ and $\gamma$) get larger and finally become stable. 
Also, the steady state value of both pest population (when they exist) are decreased 
as $\alpha $ and $\gamma$ rise. A bifurcation illustration is shown in Figure~\ref{fig3}, 
taking $\alpha$ as the main parameter. 
\begin{figure}[ht!]
\centering
\includegraphics[scale=0.64]{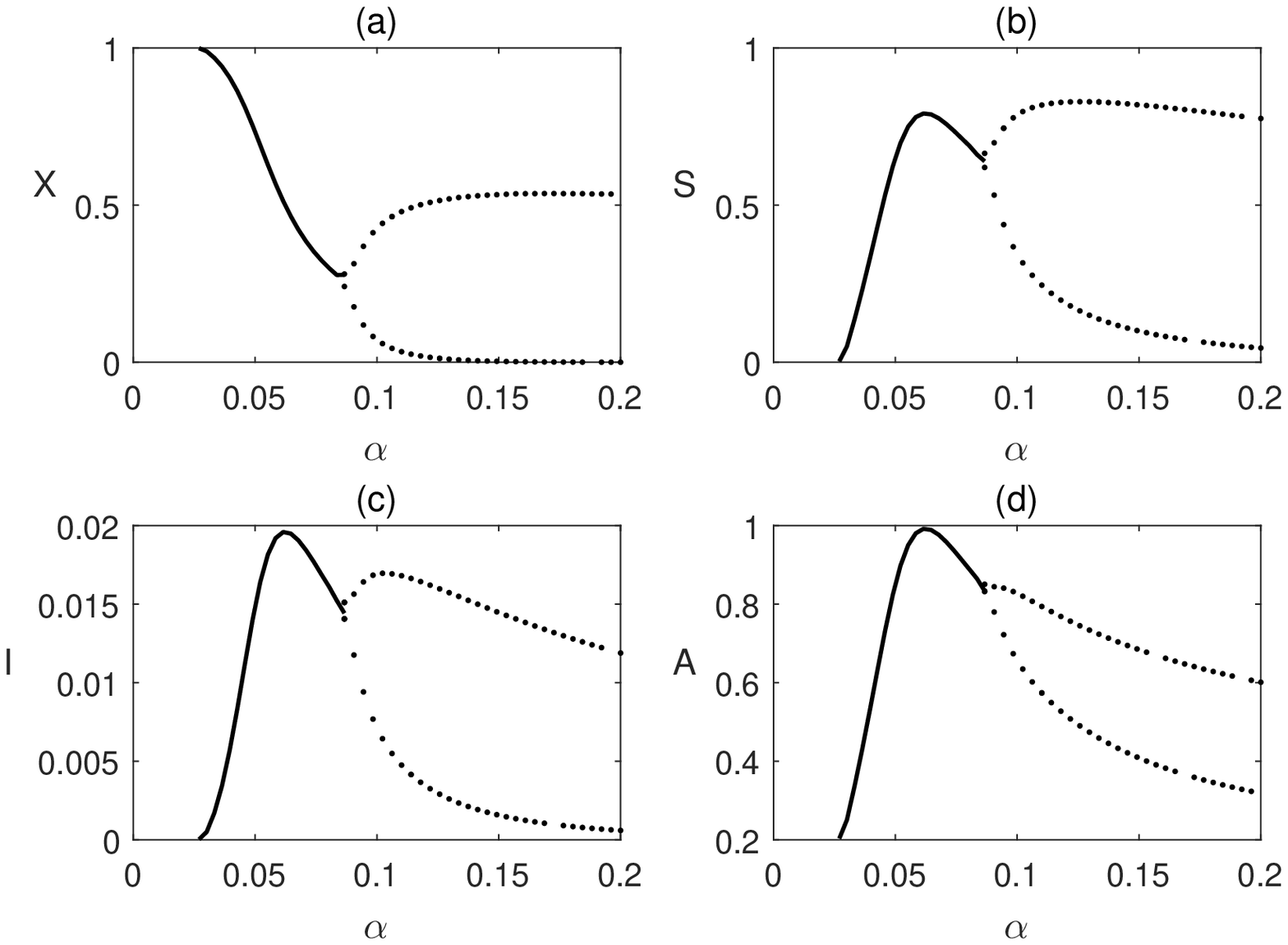}
\caption{Bifurcation diagram of the coexisting equilibrium $E^*$ 
(when exists) of system \eqref{eq1} with respect to the attack 
rate $\alpha$ of pest. Solid line indicates stable endemic equilibrium.}
\label{fig3}
\end{figure}

Critical values depend on many parameters, 
such as the conversion rates $m_{1}$ and $m_{2}$, 
the rate of the awareness program $\sigma$, 
the recruitment rate of global awareness $\omega$, 
the chemical pesticide control $u_{1}$, etc. 
We examine the impact of optimal control profiles by implementing 
a Runge--Kutta fourth-order scheme on the optimality system
\eqref{optimality}. The dynamical behavior of the model, in relation 
to the controls, is presented. The optimal policy is achieved 
by finding a solution to the state system \eqref{eq1} and costate system 
\eqref{controlsystem}. To find the optimal controls and respective states, we use the
Runge--Kutta numerical method and the technical computing program MATLAB. 
As already discussed, one needs to solve four-state equations and four adjoint equations. 
For that, first we solve system \eqref{controlsystem} with
an initial guess for the controls forward in time and then,
using the transversality conditions as initial values, 
the adjoint system \eqref{adjoint} is solved backwards 
in time using the current iteration solution of the state system. 
The controls are updated by using a convex combination of the previous
controls and the values from \eqref{ui}. The process continues
until the solution of the state equations at the present is very 
close to the previous iteration values. Precisely, in our numerical computations 
we use Algorithm~\ref{AlgFBSM}.
\begin{algorithm}[ht!]
\caption{Forward-Backward Sweep Method}
\label{AlgFBSM}
\begin{algorithmic}[1]
\State Make an initial guess for $u$ over the time interval (we took $u\equiv0$).
\State Using the initial condition $\varPhi_1=\varPhi(0)$  
and the values for $u$, solve $\varPhi$ 
\emph{forward} in time in compliance with its differential equation 
in the optimality system (we used \texttt{RK4}).
\State Using the transversality condition $\lambda_{N+1}=\lambda(t_f)$ 
and the values for $u$ and $\varPhi$, 
solve $\lambda$ \emph{backward} in time according 
to its differential equation in the optimality system (we used \texttt{RK4}).
\State Update $u$ using the new values for $\varPhi$ 
and $\lambda$ into the characterization of the optimal control.
\State Check convergence: if the variables are sufficiently close 
to the corresponding ones in the previous iteration, then output 
the current values as solutions; else return to Step 2.
\end{algorithmic}
\end{algorithm}
This algorithm solves a two point boundary-value problem, 
with divided boundary conditions at $t_{0}=0$ and $t=t_{f}$. 
The numerical solution of the optimal control problem is given 
in Figure~\ref{fig4}, showing the impact of optimal control theory.
\begin{figure}[ht!] 
\centering
\includegraphics[width=5in]{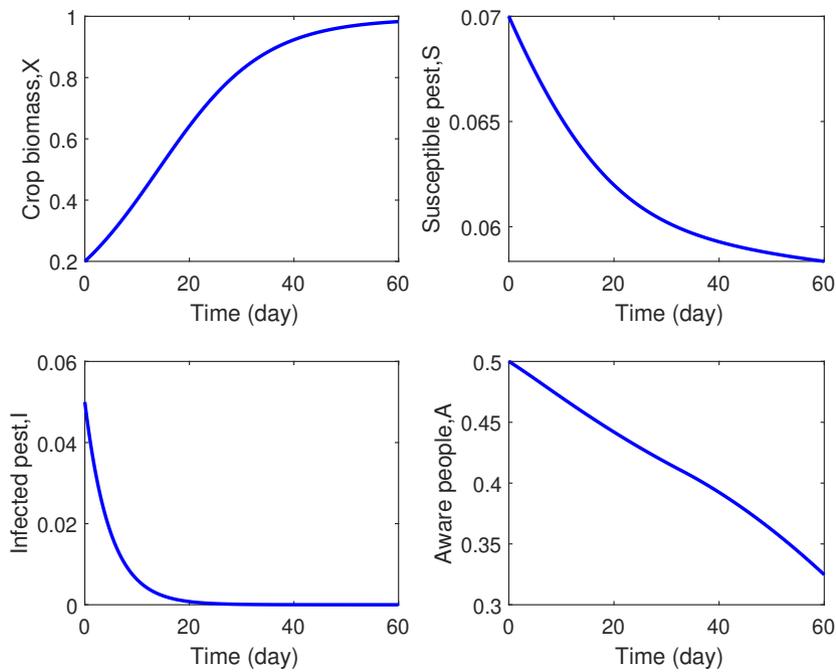}
\caption{Numerical solutions of the optimal control problem 
taking parameters as in Table~\ref{table1.1}.}
\label{fig4}
\end{figure}
We apply the control through chemical pesticide effects, bio-pesticides, 
and cost of advertisements for a time period of $60$ days. In Figure~\ref{fig4}, 
we note that, due to the effort of optimal controls \(u_{1}^*, u_{2}^*, u_{3}^*\),
crop biomass population obtains its maximum value in $60$ days, 
susceptible pest minimizes and infected pest is also minimized 
and reduced to $0$ in the first $20$ days. The population of pest is
reduced radically with an influence of the best frameworks of universal awareness 
(i.e., $u_2^*\lambda$) and chemical pesticides control movement, $u_1^*\gamma$. 
It is also seen that the susceptible pest population goes to devastation 
inside the earliest $50$ days, due to the effort of the extremal controls, 
which are shown in Figure~\ref{fig8}.
\begin{figure}[ht!]
\centering
\includegraphics[width=5in]{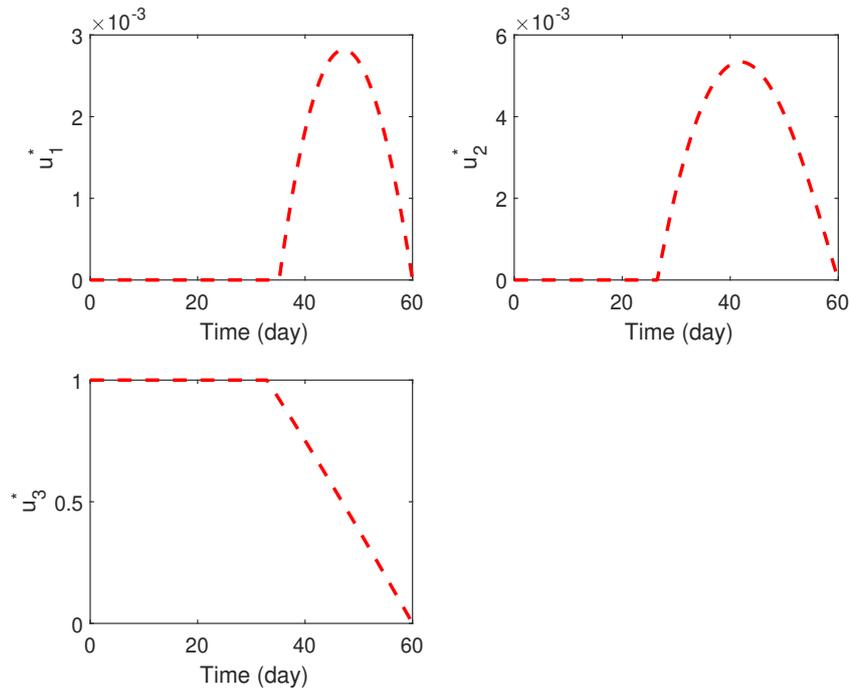}
\caption{Pontryagin extremal controls $u_1^*$,$u_2^*$ and $u_3^*$ 
plotted as functions of time.}
\label{fig8}
\end{figure}
Thus, the optimal control policy, by means of chemical pesticides, 
biological control, and global farming awareness, has a great influence 
in making the system free of pest and maintaining the stable nature 
in the remaining time period. Figure~\ref{fig8} shows that optimal chemical 
pesticides and biological control are needed to control the environmental 
crop biomass and to minimize the cost of cultivation with optimal awareness 
through global media. 


\section{Conclusions}
\label{sec9}

In this article, a mathematical model, described by a system of ordinary 
differential equations, has been developed and analyzed to plan 
the control of pests in a farming environment. Our model contains four concentrations, 
specifically, concentration of crop-biomass, density of susceptible pests, 
infected pests, and population awareness. The model under consideration exhibits 
four feasible steady state points: the crop-pest free equilibrium point, 
which is unstable for all parameter values; the pest free equilibrium point; 
the susceptible pest free equilibrium point, which may exist when 
the carrying capacity $K$ is greater than the crop biomass $X$; 
and the interior equilibrium point. Local stability of the positive interior 
equilibrium point \(E^{*}\) and local Hopf-bifurcation around it have been studied.
We have shown how the dynamics changes with the parameter value \(\alpha\)
(the consumption rate of pest to crops). The dynamical behavior of the system 
was investigated using stability theory, optimal control theory, and numerical simulations.
We assumed that responsive groups take on bio-control, such as the included pest managing, 
as it is eco-friendly and is fewer injurious to individual health and surroundings. 
Neighboring awareness movements may be full as comparative to the concentration 
of susceptible pest available in the crop biomass. We expect that the international 
issues, disseminated by radio, TV, telephone, internet, etc., enlarge the stage 
of consciousness. Moreover, we have used optimal control theory to provide 
the price effective outline of bio-pesticides, chemical pesticide costs and 
a universal alertness movement. We observed the dynamical behavior of 
the controlled system and the effects of the three controls. 
This work can be extended in several ways, for example 
by introducing time delays in the awareness level of farmers attitude 
towards observation of fields and in becoming aware of their farm after 
campaigns made. Consideration of the crop population as infected and uninfected 
cases is also another possible extension to the present paper, 
in order to enrich the proposed mathematical model for pest control.


\section*{Acknowledgments}

This work is part of first authors' PhD project, carried out
at Adama Science and Technology University (ASTU), Ethiopia.
Abraha acknowledges ASTU for its welcome and bear during this work, 
through the research grant ASTU/SP-R/027/19. 
Torres is grateful to the financial support from the Portuguese Foundation 
for Science and Technology (FCT), through CIDMA and project UIDB/04106/2020.



\section*{Author Biographies}


\begin{biography}{\includegraphics[scale=0.06]{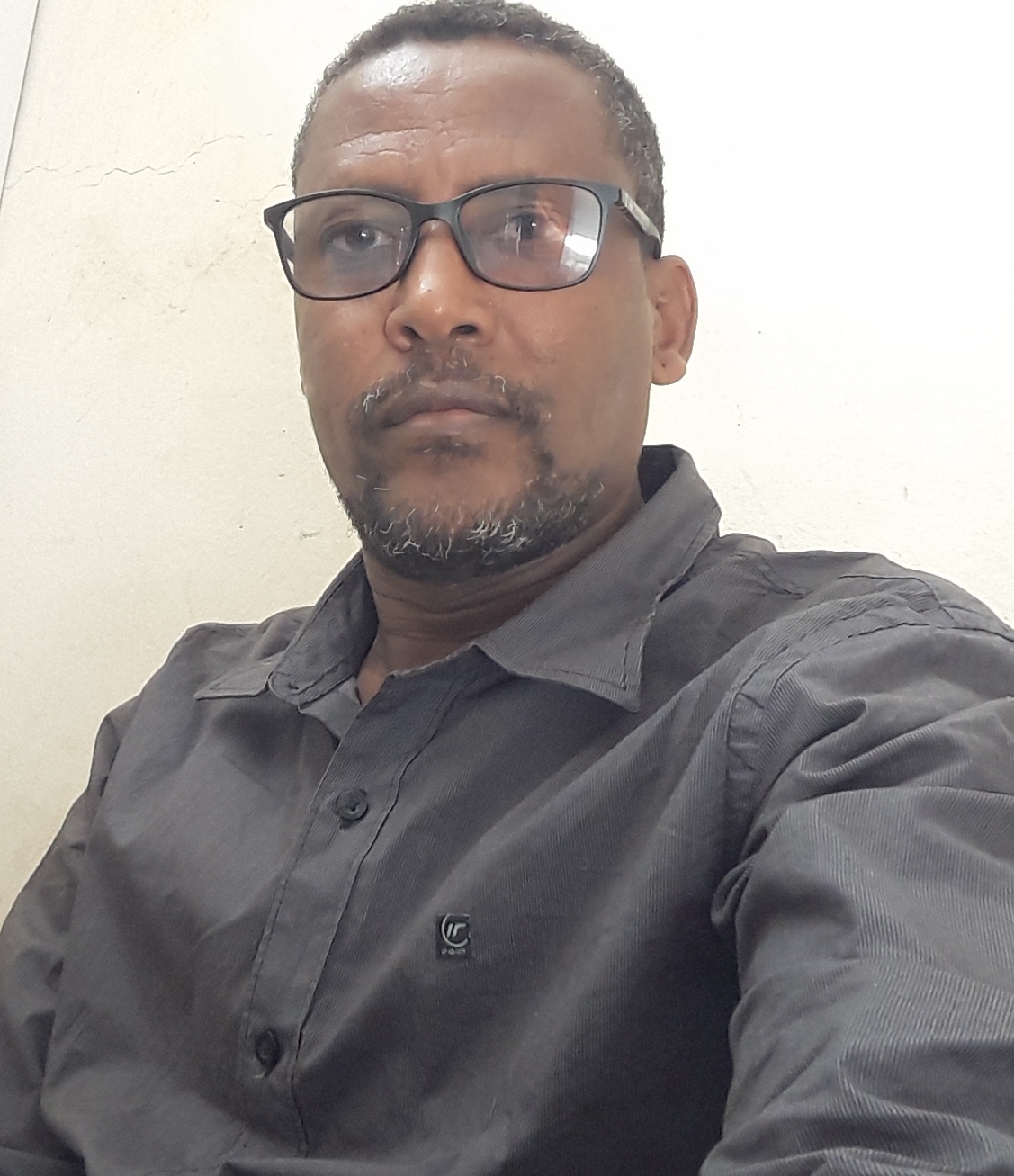}}{\textbf{Teklebirhan
Abraha} received his BSc degree in Applied Mathematics from the University of Gondar, Ethiopia, 
in July 2006, and an MSc degree in Applied Mathematics (Optimization) from Addis Ababa University,
Ethiopia, in January 2011. He is currently a PhD candidate in Applied Mathematics (Optimization) 
at Adama Science and Technology University, Ethiopia, under the supervision 
of Professor Delfim F. M Torres and Dr. Legesse Lemecha. His research interests 
are in the areas of applied mathematics, including optimization and optimal control, 
mathematical modeling of biological systems, and operations research.}
\end{biography}

\begin{biography}{\includegraphics[scale=0.27]{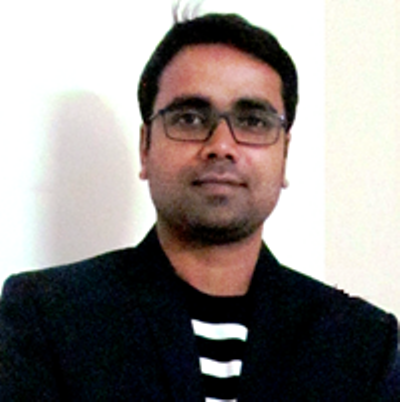}}{\textbf{Fahad Al Basir} 
is an Assistant Professor, Department of Mathematics, Asansol Girls' College, West Bengal, India. 
He received BSc, MSc, and PhD degrees from Jadavpur University, Kolkata, India. 
He joined, as a Post-Doctoral fellow at the Department of Zoology, Visva-Bharati University, 
Santiniketan, India. He received Dr. D.S. Kothari Post-Doctoral Fellowship, University Grants
Commission, from the Government of India. He is serving as an academic editor of 
Modelling and Simulation in Engineering, an Hindawi publication. He has authored 
and co-authored several research articles in reputed journals. His research includes 
mathematical modeling using ordinary and delay differential equations 
in disease and pest management, chemical and biochemical systems, and ecology.}
\end{biography}

\begin{biography}{\includegraphics[scale=0.47]{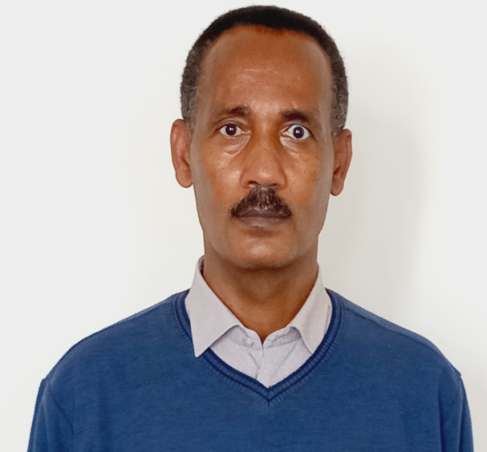}}{\textbf{Legesse Lemecha Obsu} 
is an Associate Professor of Mathematics and a dean of Postgraduate program 
at Adama Science and Technology University, Ethiopia.  He received MSc and PhD 
degrees in Mathematics from Addis Ababa University. From 1995 to 1999 
he was an undergraduate student at the then Kotebe College of Teachers Education. 
He has authored and co-authored several research articles in reputed journals. 
His area of research is mainly focused on Mathematical Modeling, 
including traffic flow, epidemiology (infectious diseases) and ecology.}
\end{biography}

\begin{biography}{\includegraphics[scale=0.23]{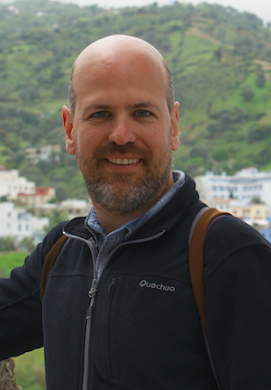}}{\textbf{Delfim F. M. Torres} 		
is a Portuguese Mathematician born 16 August 1971 in Nampula, Portuguese Mozambique. 
He obtained a PhD in Mathematics from University of Aveiro (UA) in 2002, 
and Habilitation in Mathematics, UA, in 2011. He is a Full Professor 
of Mathematics, since 9 March 2015, Director of the R\&D Unit CIDMA, 
the largest Portuguese research center in Mathematics, 
and Coordinator of its Systems and Control Group. His main research area 
is calculus of variations and optimal control; optimization; fractional derivatives 
and integrals; dynamic equations on time scales; and mathematical biology. 
Torres has written outstanding scientific and pedagogical publications. In particular, 
he is author of two books with Imperial College Press and three books with Springer. 
He has strong experience in graduate and post-graduate student supervision 
and teaching in mathematics. Twenty PhD students in Mathematics 
have successfully finished under his supervision. Moreover, he has been team leader and member 
in several national and international R\&D projects, including EU projects 
and networks. Prof. Torres is a Highly Cited Researcher in Mathematics, 
having been awarded the title in 2015, 2016, 2017, and 2019.
He is, since 2013, the Director of the Doctoral Programme 
Consortium in Applied Mathematics (MAP-PDMA) of Universities 
of Minho, Aveiro, and Porto. Delfim is married since 2003, 
and has one daughter and two sons.}
\end{biography}


\end{document}